\documentclass[10pt]{amsart}
\usepackage{setspace}
\usepackage{amsmath, amsfonts, amsthm, amstext, xspace}
\usepackage{amssymb,latexsym}
\usepackage{xcolor}
\usepackage{hyperref}
\definecolor{seagreen}{RGB}{46,139,87}
\definecolor{maroon}{RGB}{128,0,0}
\definecolor{darkviolet}{RGB}{148,0,211}
\definecolor{twelve}{RGB}{100,100,170}
\definecolor{thirteen}{RGB}{100,150,50}
\definecolor{fourteen}{RGB}{200,0,0}
\definecolor{fifteen}{RGB}{0,200,0}
\definecolor{sixteen}{RGB}{0,0,200}
\definecolor{seventeen}{RGB}{200,0,200}
\definecolor{eighteen}{RGB}{0,200,200}

\usepackage{graphicx}
\usepackage{lscape}

\makeatletter
\newcommand{\addresseshere}{%
  \enddoc@text\let\enddoc@text\relax
}
\makeatother

\usepackage[all]{xy}
\newtheorem{thm}{Theorem}[section]
\newtheorem{lem}[thm]{Lemma}

\newtheorem{prop}[thm]{Proposition}

\newtheorem{rem2}[thm]{Remark}

\newtheorem{conv}[thm]{Convention}
\theoremstyle{definition}
\newtheorem{defin}[thm]{Definition}
\newtheorem{exm}[thm]{Example}

\def\f{\mathbb{F}}

\def\p{\mathbb{P}}

\def\r{\mathbb{R}}

\def\z{\mathbb{Z}}

\newcommand{\modmod}{\mathbin{\!/\mkern-5mu/\!}} 

\usepackage{amssymb}
\usepackage{tikz-cd}

\author{J.D. Quigley}\address{Cornell University}\email{jdq27@cornell.edu}
\title[$\mathit{tmf}$-based Mahowald invariants]{$\mathit{tmf}$-based Mahowald invariants}

\begin{document}
\maketitle

\begin{abstract}
The $2$-primary homotopy $\beta$-family, defined as the collection of Mahowald invariants of Mahowald invariants of $2^i$, $i \geq 1$, is an infinite collection of periodic elements in the stable homotopy groups of spheres. In this paper, we calculate $\mathit{tmf}$-based approximations to this family. Our calculations combine an analysis of the Atiyah-Hirzebruch spectral sequence for the Tate construction of $\mathit{tmf}$ with trivial $C_2$-action and Behrens' filtered Mahowald invariant machinery.
\end{abstract}

\tableofcontents

\section{Introduction}\label{Sec:Introduction}

\subsection{Greek letter families}

\emph{Greek letter elements} form infinite, periodic families in the stable homotopy groups of spheres. These families were constructed at low chromatic heights  by Adams \cite{Ada66}, Smith \cite{Smi70}, and Toda \cite{Tod71}, and their work led to the Greek letter construction which builds analogous families at higher chromatic heights. This general procedure is effective in many cases, but it requires the existence of certain generalized Moore spectra. These do not always exist, so the number of Greek letter families which can be produced via the Greek letter construction is limited. 

Miller, Ravenel, and Wilson \cite{MRW77} defined \emph{algebraic Greek letter elements} in the $E_2$-term of the Adams-Novikov spectral sequence for the sphere. They then defined Greek letter elements as the classes in homotopy which the algebraic Greek letter elements detect. This definition works at all chromatic heights and all primes, but since algebraic Greek letter elements do not always survive in the Adams-Novikov spectral sequence, some of the resulting Greek letter elements are zero.

In \cite{MR93}, Mahowald and Ravenel defined \emph{homotopy Greek letter elements}. These elements are defined at all chromatic heights and all primes, and moreover, they are always nonzero. Furthermore, calculations at low heights suggest that homotopy Greek letter elements coincide with the Greek letter elements as defined above (whenever they are nonzero). Homotopy Greek letter elements have been completely calculated at chromatic height one, and they have been studied extensively at chromatic height two in the odd-primary setting. They have also been calculated in low dimensions at the prime two. This paper is a first step towards completely calculating the $2$-primary homotopy Greek letter family at chromatic height two.

\subsection{The Mahowald invariant and homotopy Greek letter families}\label{Sec:MIGreek}

Homotopy Greek letter elements are defined using the Mahowald invariant. Although our interest in the Mahowald invariant is limited to chromatic homotopy theory in this paper, we note that the Mahowald invariant has a wide array of applications, including unstable homotopy theory \cite{Mah67, MR93}, geometry \cite{Sch82, Sto88, HLSX18}, equivariant homotopy theory \cite{BG95}, and motivic homotopy theory \cite{Qui19a, Qui19c, Qui19b}. 

The Mahowald invariant of a class $\alpha \in \pi_t(S^0)_{(p)}$ in the $p$-local stable stems is a nontrivial coset $M(\alpha) \subset \pi_{t+N-1}(S^0)_{(p)}$ contained in a higher $p$-local stable stem. Let $\alpha \in \pi_t(S^0)_{(2)}$. The \emph{Mahowald invariant of $\alpha$}, denoted $M(\alpha)$, is the coset of completions of the diagram
\begin{equation}\label{Eqn:MI}
\begin{tikzcd}
S^t \arrow{d}{\alpha} \arrow[dashed,rr,"M(\alpha)"] & & S^{-N+1} \arrow{d} \\
S^0 \arrow{r}{\simeq} & \Sigma P^\infty_{-\infty} \arrow{r} & \Sigma P^\infty_{-N}
\end{tikzcd}
\end{equation}
where $N>0$ is the minimal integer such that the left-hand composite from $S^t$ to $\Sigma P^\infty_{-N}$ is nontrivial. The spectrum $P^\infty_{-N}$ is the Thom spectrum of the $(-N)$-fold Whitney sum of the tautological bundle over $\r\p^\infty$, and $P^\infty_{-\infty}$ is the inverse limit of these as $N$ tends to infinity. In particular, the minimality of $N$ ensures that the coset $M(\alpha)$ is nontrivial. The map $S^0 \to \Sigma P^\infty_{-\infty}$ is a $2$-adic equivalence by Lin's Theorem \cite{LDMA80}.

In \cite[Conj. 12]{MR87}, Mahowald and Ravenel conjectured that the Mahowald invariant carries $v_n$-periodic classes to $v_n$-torsion classes (with some exceptions). This conjecture has been verified by explicit computation in many cases:
\begin{enumerate}
\item When $n=0$ and $p=2$, Mahowald and Ravenel computed $M(2^i)$ for all $i \geq 1$ and showed all the elements in it are $v_1$-periodic \cite{MR93}.
\item When $n=0$ and $p\geq 3$, Mahowald and Ravenel and Sadofsky showed $\alpha_i \in M(p^i)$ for all $i \geq 1$ \cite{MR93, Sad92}.
\item When $n=1$ and $p\geq 5$, Mahowald and Ravenel and Sadofsky proved $\beta_i \in M(\alpha_i)$ for all $i \geq 1$ \cite{MR93, Sad92}.
\item When $n=1$ and $p \geq 5$, Sadofsky further proved that $\beta_{p/2} \in M(\alpha_{p/2})$ \cite{Sad92}.
\item When $n=1$ and $p =3$, Behrens calculated that $(-1)^{i+1} \beta_i \in M(\alpha_i)$ for $i \equiv 0,1,5 \mod 9$ \cite{Beh06}.
\item When $n=1$ and $p=2$, Behrens determined $M(x)$ for $x \in \pi_{\leq 12}(S^0)_{(2)}$. In particular, his calculations imply that $M(M(2^i))$ is $v_2$-periodic for $i \leq 7$ \cite{Beh07}. 
\end{enumerate}

In all of these calculations, (iterated) Mahowald invariants contain Greek letter elements whenever they exist and are nonzero \cite{MR93}. This observation led Mahowald and Ravenel to make the following definition. 

\begin{defin}\label{Def:HPB}\cite[Def. 3.6]{MR93}
The \emph{$p$-primary $i$-th homotopy Greek letter element} is defined by $\alpha_i^h := M(p^i)$, $\beta_i^h := M(M(p^i))$, and so on. 
\end{defin}

The computations of Mahowald, Ravenel, Sadofsky, and Behrens may therefore be viewed as computations of $\alpha_i^h$ at all primes, $\beta^h_i$ at all primes $p \geq 5$,  $\beta^h_i$ for $i \equiv 0,1,5 \mod 9$ at the prime $p=3$, and $\beta^h_i$ for $i \leq 7$ at the prime $p=2$. 

\subsection{The $E$-based Mahowald invariant}

The Mahowald invariant is difficult to calculate directly: one must understand a range of the stable homotopy groups of $P^\infty_{-N}$ for various $N>0$, and this is roughly as difficult as understanding the stable homotopy groups of spheres (cf. \cite{CLM88, DM86}). Therefore most Mahowald invariant calculations, including all of those listed above, rely on certain approximations to the Mahowald invariant which use simpler homology theories than stable homotopy groups. 

Let $E$ be a spectrum equipped with a trivial $C_2$-action. By \cite[Thm. 16.1]{GM95}, there is an equivalence
$$E^{tC_2} \simeq \lim_n (\Sigma P^\infty_{-n} \wedge E),$$
where $E^{tC_2}$ is the $C_2$-Tate construction. Let $\alpha \in \pi_t(E^{tC_2})$. The \emph{$E$-based Mahowald invariant of $\alpha$}, $M_E(\alpha)$, is the coset of completions of the diagram
\[
\begin{tikzcd}
S^t \arrow{d}{\alpha} \arrow[dashed,rr,"M_E(\alpha)"] & & \Sigma^{-N+1} E \arrow{d} \\
E^{tC_2} \arrow{r}{\simeq} &  \underset{\leftarrow}{\lim} (E \wedge \Sigma P^\infty_{-n}) \arrow{r} & E \wedge \Sigma P^\infty_{-N}
\end{tikzcd}
\]
where $N>0$ is minimal such that the left-hand composite from $S^t$ to $E \wedge \Sigma P^\infty_{-N}$ is nontrivial. 

Mahowald and Ravenel \cite{MR93} and Sadofsky \cite{Sad92} used the $BP$-based Mahowald invariant to compute $M(p^i)$, $p \geq 3$, and $M(M(p^i))$, $p \geq 5$. However, $BP$ cannot be used to calculate $M(2^i)$ using the same techniques because $BP$ does not detect the relevant elements in the $2$-primary stable stems. Instead, Mahowald and Ravenel calculated $M(2^i)$ using the $bo$-based Mahowald invariant. In more detail, Davis and Mahowald \cite{DM84} showed that there is an equivalence after $2$-completion
$$bo^{tC_2} \simeq \bigvee_{i \in \z} \Sigma^{4i} H\z_2,$$
so in particular $2^i \in \z_2 \cong \pi_0(\bigvee_{i \in \z} H\z_2) \cong \pi_0(bo^{tC_2})$ for all $i \geq 1$. Thus $M_{bo}(2^i)$ is a well-defined coset in $\pi_*(bo)$. Mahowald and Ravenel computed this for all $i \geq 1$, then lifted their results along the Hurewicz map $S^0 \to bo$ to determine $\alpha_i^h = M(2^i)$ for all $i \geq 1$.

With this example in mind, the calculation of $M(M(2^i))$ should use the height two analog of $bo$, connective topological modular forms $\mathit{tmf}$. Bailey and Ricka \cite{BR19} showed that there is an equivalence after $2$-completion
$$\mathit{tmf}^{tC_2} \simeq \bigvee_{i \in \z} \Sigma^{8i} bo,$$
so $M_{\mathit{tmf}}(\alpha)$ is well-defined for any $\alpha \in \pi_*(bo)$. Thus $M_{\mathit{tmf}}(M_{bo}(2^i))$ is a well-defined coset in $\pi_*(\mathit{tmf})$. Mahowald and Ravenel's calculations at chromatic height one suggest that at height two, $M(M(2^i))$ may be determined by calculating $M_{\mathit{tmf}}(M_{bo}(2^i))$ for all $i \geq 1$ and then lifting these calculations along the Hurewicz map $S^0 \to \mathit{tmf}$. 

\subsection{Statement of main result}

In this paper, we carry out the first step in the program outlined above by computing $M_{\mathit{tmf}}(M_{bo}(2^i))$ for all $i \geq 1$. In Theorem \ref{tmfbasedMI}, the class $\eta$ generates $\pi_1(bo) \cong \z/2$, the class $\alpha$ generates $\pi_4(bo) \cong \z_2$, and the class $\beta$ is the Bott element which generates $\pi_8(bo) \cong \z_2$. The elements $M_{\mathit{tmf}}(M_{bo}(x)) \in \pi_*(\mathit{tmf})$ are named as in \cite{DFHH14}. The $bo$-based Mahowald invariants $M_{bo}(2^i)$ were calculated by Mahowald and Ravenel in \cite{MR93}. 

\begin{thm}\label{tmfbasedMI}
The following tables consist of the classes $2^i \in \pi_*(H\z_2)$, their $bo$-based Mahowald invariants $M_{bo}(2^i) \in \pi_*(bo)$, and the $\mathit{tmf}$-based Mahowald invariants of their $bo$-based Mahowald invariants $M_{\mathit{tmf}}(M_{bo}(2^i))$. The remaining $\mathit{tmf}$-based Mahowald invariants are determined by the rule $M_{\mathit{tmf}}(M_{bo}(2^{32} \cdot x)) = M_{\mathit{tmf}}(\beta^8 x) = \Delta^8 M_{\mathit{tmf}}(x)$. 
\begin{center}
\begin{tabular}{ |c|c|c|}
\hline
$x$ & $M_{bo}(x)$ & $M_{\mathit{tmf}}(M_{bo}(x))$ \\
\hline
$2^0$ & $1$ & $1[0]$\\
$2^1$ & $\eta$ & $\nu[-2]$ \\
$2^2$ & $\eta^2$ & $\nu^2[-4]$ \\
$2^3$ & $\alpha$ & $(c_4 + \epsilon)[-4]$ \\
\hline
\end{tabular}
\begin{tabular}{ |c|c|c|}
\hline
$x$ & $M_{bo}(x)$ & $M_{\mathit{tmf}}(M_{bo}(x))$ \\
\hline
$2^4$ & $\beta$ & $\bar{\kappa}[-12]$ \\
$2^5$ & $\beta \eta$ & $\eta \Delta[-16]$ \\
$2^6$ & $\beta \eta^2$ & $\bar{\kappa} \epsilon[-18]$ \\
$2^7$ & $\beta \alpha$ & $q[-20]$ \\
\hline
\end{tabular}
\begin{tabular}{ |c|c|c|}
\hline
$x$ & $M_{bo}(x)$ & $M_{\mathit{tmf}}(M_{bo}(x))$ \\
\hline
$2^8$ & $\beta^2$ & $\eta \Delta \bar{\kappa}[-29]$ \\
$2^9$ & $\beta^2 \eta$ & $\eta^2 \Delta^2[-33]$ \\
$2^{10}$ & $\beta^2 \eta^2 $ & $\nu \Delta^2 \eta^2[-35]$ \\
$2^{11}$ & $\beta^2 \alpha$ & $(c_4 + \epsilon) \Delta^2[-36]$ \\
\hline
\end{tabular}
\end{center}

\begin{center}
\begin{tabular}{ |c|c|c|}
\hline
$x$ & $M_{bo}(x)$ & $M_{\mathit{tmf}}(M_{bo}(x))$ \\
\hline
$2^{12}$ & $\beta^3$ & $\eta\Delta\bar{\kappa}^2[-41]$ \\
$2^{13}$ & $\beta^3 \eta$ & $\eta^2 \Delta^2 \bar{\kappa}[-45]$ \\
$2^{14}$ & $\beta^3 \eta^2$ & $\eta^3 \Delta^3[-49]$ \\
$2^{15}$ & $\beta^3 \alpha$ & $\Delta^3 c_4[-52]$ \\
\hline
\end{tabular}
\begin{tabular}{ |c|c|c|}
\hline
$x$ & $M_{bo}(x)$ & $M_{\mathit{tmf}}(M_{bo}(x))$ \\
\hline
$2^{16}$ & $\beta^4$ & $\eta^2\Delta^2\bar{\kappa}^2[-58]$\\
$2^{17}$ & $\beta^4\eta$ & $\nu\Delta^4[-66]$ \\
$2^{18}$ & $\beta^4\eta^2$ & $\Delta^4\nu^2[-68]$ \\
$2^{19}$ & $\beta^4\alpha$ & $\Delta^4(c_4 + \epsilon)[-68]$ \\
\hline
\end{tabular}
\begin{tabular}{ |c|c|c|}
\hline
$x$ & $M_{bo}(x)$ & $M_{\mathit{tmf}}(M_{bo}(x))$ \\
\hline
$2^{20}$ & $\beta^5$ & $2\Delta^4 \bar{\kappa}[-76]$ \\
$2^{21}$ & $\beta^5 \eta$ & $\eta^2 \Delta^5[-81]$ \\
$2^{22}$ & $\beta^5 \eta^2$ & $\eta \Delta \bar{\kappa}^5[-83]$ \\
$2^{23}$ & $\beta^5 \alpha$ & $\Delta^4q[-84]$ \\
\hline
\end{tabular}
\end{center}

\begin{center}
\begin{tabular}{ |c|c|c|}
\hline
$x$ & $M_{bo}(x)$ & $M_{\mathit{tmf}}(M_{bo}(x))$ \\
\hline
$2^{24}$ & $\beta^6$ & $\Delta^4 \kappa^3[-90]$ \\
$2^{25}$ & $\beta^6 \eta$ & $\eta^2 \Delta^5 \bar{\kappa}[-93]$ \\
$2^{26}$ & $\beta^6 \eta^2 $ & $\Delta^6 \eta^3[-97]$ \\
$2^{27}$ & $\beta^6 \alpha$ & $(c_4 + \epsilon) \Delta^6[-100]$ \\
\hline
\end{tabular}
\begin{tabular}{ |c|c|c|}
\hline
$x$ & $M_{bo}(x)$ & $M_{\mathit{tmf}}(M_{bo}(x))$ \\
\hline
$2^{28}$ & $\beta^7$ & $\nu \Delta^6 \kappa[-105]$ \\
$2^{29}$ & $\beta^7 \eta$ & $\nu\Delta^6 \kappa \eta[-105]$ \\
$2^{30}$ & $\beta^7 \eta^2$ & $\nu \Delta^6 \kappa \nu[-107]$ \\
$2^{31}$ & $\beta^7 \alpha$ & $\Delta^7 c_4[-116]$ \\
\hline
\end{tabular}
\end{center}
\end{thm}

\begin{rem2}
In low dimensions, our calculations suggest that $M_{\mathit{tmf}}(M_{bo}(2^i))$ is a close approximation to $M(M(2^i))$. In particular, we see that $M_{\mathit{tmf}}(M_{bo}(2^i))$ is in the Hurewicz image of $\mathit{tmf}$ for $i \in \{1,2,6\}$ and it agrees with $M(M(2^i))$ as calculated in \cite{Beh07}. 
\end{rem2}

Theorem \ref{tmfbasedMI} is proven in Section \ref{SectionFiltered}. We will say more about the proof in the next section. 

\subsection{Techniques: The Atiyah-Hirzebruch spectral sequence and filtered Mahowald invariant}

The $E$-based Mahowald invariant is typically computed by analyzing the Atiyah-Hirzebruch spectral sequence (AHSS) converging to $\pi_*(E^{tC_2})$. For $E = bo$, this is a fairly simple computation since the AHSS collapses at $E_4$. The computation is not straightforward when $E = \mathit{tmf}$, and we do not know precisely when the AHSS collapses. We analyze the AHSS up to the $E_8$-page in Section \ref{AHSS}, but this partial analysis does not suffice to calculate $M_{\mathit{tmf}}(M_{bo}(2^i))$ for all $i \geq 1$. The necessary additional machinery is described below.

Lin's Theorem provides an interesting filtration of the stable stems by attaching to each class $\alpha \in \pi_*(S^0)$ the dimension of the cell of $\Sigma RP^\infty_{-\infty}$ where it is detected. In practice, calculating the filtration of $\alpha$ is essentially equivalent to computing the Mahowald invariant of $\alpha$. If $E$ is a ring spectrum, then the $E$-Adams filtration gives an alternative method for attaching a number to $\alpha$; namely, the stage of the $E$-Adams resolution of the sphere where $\alpha$ is detected \cite[Ch. 2]{Rav03}. In \cite{Beh06}, Behrens combined these two filtrations of $\pi_*(S^0)$ to obtain a bifiltration which he used to define \emph{$E$-filtered Mahowald invariants}, $^EM^{[k]}(\alpha)$, $k \geq 0$. These formalize the idea of computing ``Mahowald invariants up to $E$-Adams filtration $k$." In particular, one has $M(\alpha) = {^E}M^{[\infty]}(\alpha)$ if the $E$-based Adams spectral sequence (ASS) converges, and if the spectral sequence has a vanishing line of finite slope, then one has $M(\alpha) = {^EM}^{[k]}(\alpha)$ for all $k >>0$. 

Behrens provides a procedure for lifting filtered Mahowald invariants into higher Adams filtration \cite[Procedure 9.1]{Beh06}. Roughly speaking, given an $E$-filtered Mahowald invariant $^EM^{[k]}(\alpha)$, an algorithmic analysis of the $E$-based Adams spectral sequence (ASS) and the attaching maps in $\Sigma P^\infty_{-\infty}$ allow one to compute ${^EM}^{[k+1]}(\alpha)$. This procedure was used to great effect in \cite{Beh06} and \cite{Beh07} to compute Mahowald invariants of $v_1$-periodic classes at the primes $p=2,3$. We note that to start the procedure, one must also identify the first nontrivial filtered Mahowald invariant; this can usually be done by consulting existing $Ext$-calculations such as \cite{Bru98}. 

To compute $M_{\mathit{tmf}}(M_{bo}(2^i))$ for all $i \geq 1$, we adapt Behrens' filtered Mahowald invariant machinery to the category of $\mathit{tmf}$-modules. In particular, we define the \emph{$E$-filtered $\mathit{tmf}$-based Mahowald invariant} by combining the $E$-Adams filtration of $\mathit{tmf}$ with the filtration of $\mathit{tmf}^{tC_2} \simeq \underset{\leftarrow}{\lim} (\mathit{tmf} \wedge \Sigma P^\infty_{-\infty})$ induced by the cellular filtration of $\Sigma P^\infty_{-\infty}$. We provide a pseudo-algorithm for lifting $E$-filtered $\mathit{tmf}$-based Mahowald invariants into higher Adams filtration in Section \ref{Sec:Pseudoalgorithm}. As in Behrens' calculations, we consult existing $Ext$-calculations due to Davis and Mahowald \cite{DM82} to determine the first nontrivial $H\f_2$-filtered $\mathit{tmf}$-based Mahowald invariants. We summarize their computations in Section \ref{SectionAlg}. 

\subsection{Ongoing and future work}
We conclude by mentioning some ongoing and future work related to the calculations in this paper. 

\subsubsection{Finite complexes and Greek letter families}
In \cite[Def. 3.1]{Beh07}, Behrens proposed a different definition of the homotopy Greek letter elements:

\begin{defin}[Behrens]
Suppose that $X$ is a type $n$ $p$-local finite complex for which $BP_*(X)$ is a free module over $BP_*/I$, for $I = (p^{i_0}, v_1^{i_1},\ldots,v_{n-1}^{i_{n-1}})$. Suppose that $X$ has $v_n^k$-multiplication. Then we define the homotopy Greek letter element $(\alpha^{(n)})^h_{k/i_{n-1},\ldots,i_0}$ to be the element of $\pi_*(S^0)$ which detects $v_n^k \in \pi_*(X)$ in the $E_1$-term of the AHSS.
\end{defin}

This definition may depend on the choice of $X$ and choice of detecting elements in the AHSS in \cite[Rem. 3.2]{Beh07}. However, if one takes $X = S^0/(2,\eta)$ to be a finite type $1$ complex with a $v_1^1$-self-map, then the resulting elements $\alpha^h_{i/1}$, $i \geq 1$, coincide with the family elements $\alpha^h_i := M(2^i)$, $i \geq 1$, defined using the Mahowald invariant. This suggests that if one takes $X$ to be a type $n$ complex with a $v_n^1$-self-map, then the finite complex definition of the homotopy Greek letter elements may produce the same elements as the Mahowald invariant definition.

In \cite{BE16}, Bhattacharya and Egger define a class of finite spectra which admit a $v_2^1$-self-map. In work in progress with Bhattacharya, we are calculating $\mathit{tmf}$-based approximations to the finite complex definition of $\beta^h_i$ for all $i \geq 1$. We plan to compare our finite complex calculations to the Mahowald invariant calculations in Theorem \ref{tmfbasedMI}.

\subsubsection{The $2$-primary homotopy $\beta$-family}

As noted above, the $\mathit{tmf}$-based Mahowald invariants we calculated are approximations to the homotopy $\beta$-family at the prime two. In future work, we plan to lift the computations in Theorem \ref{tmfbasedMI} (and their finite complex analogs calculated with Bhattacharya) along the Hurewicz map $S^0 \to \mathit{tmf}$ using the $\mathit{tmf}$-resolution. 

\subsection{Outline}
In Section \ref{AHSS}, we analyze the AHSS for $\mathit{tmf}^{tC_2}$. The resulting $E_8$-page is depicted in Appendix \ref{App:AHSS}.

In Section \ref{SectionAlg}, we introduce the algebraic $E$-based Mahowald invariant and review Davis and Mahowald's algebraic $\mathit{tmf}$-based Mahowald invariant calculations from \cite{DM82}. These algebraic computations serve as the starting point for our calculations of $\mathit{tmf}$-based Mahowald invariants. 

In Section \ref{SectionFiltered}, we adapt Behrens' filtered Mahowald invariant machinery to the category of $\mathit{tmf}$-modules. We apply this to compute $M_{\mathit{tmf}}(M_{bo}(2^i))$ for all $i \geq 1$. The result is summarized in Theorem \ref{tmfbasedMI}.

\subsection{Conventions}

Unless otherwise stated, everything outside of the Introduction is implicitly $2$-complete. We will write `AHSS' for ``Atiyah-Hirzebruch spectral sequence" and ``Atiyah-Hirzebruch spectral sequence converging to $\pi_*(\mathit{tmf}^{tC_2})$" and `ASS' for ``Adams spectral sequence." 

The Mahowald invariant and $E$-based Mahowald invariant is a coset of elements, but we will write ``$\beta = M_{tmf}(\alpha)$" instead of ``$\beta$ is contained in $M_E(\alpha)$" for conciseness. 

\subsection{Acknowledgements.} The author thanks Mark Behrens for his guidance throughout this project, as well as Prasit Bhattacharya and a referee for helpful comments. The author also thanks Andr{\'e} Henriques for allowing the use and modification of the picture of $\pi_*(\mathit{tmf})$  from \cite{DFHH14} in an earlier version. Finally, the author thanks Tilman Bauer and Hood Chatham, whose spectral sequence programs were used to produce various figures. The author was partially supported by NSF grant DMS-1547292.

\section{The Atiyah-Hirzebruch spectral sequence for $\mathit{tmf}^{tC_2}$}\label{AHSS}
We begin by analyzing the AHSS which arises from the cellular filtration of $\Sigma P^\infty_{-\infty}$ and converges to $\pi_*(\mathit{tmf}^{tC_2})$. Differentials in this spectral sequence are induced by the attaching maps in $P^\infty_{-n}$ which are detected by $\mathit{tmf}$. It is generally hard to determine such attaching maps unless they are detected by primary squaring operations, so we cannot completely determine the differentials in this spectral sequence. Instead, we run just enough differentials to make our $\mathit{tmf}$-based Mahowald invariant calculations in Section \ref{SectionFiltered}.

\subsection{The Atiyah-Hirzebruch spectral sequence}
The homotopy groups of $\mathit{tmf}$ were computed at $p=2$ by Bauer in \cite{Bau08}, and the homotopy groups of the Tate construction of $\mathit{tmf}$ equipped with a trivial $C_2$-action were computed by Bailey and Ricka:

\begin{thm}\cite[Theorem 1.1]{BR19}
There is an equivalence of spectra
$$\mathit{tmf}^{tC_2} \simeq \prod_{i \in \z} \Sigma^{8i} bo.$$
\end{thm}

By \cite[Thm. 16.1]{GM95}, the Tate construction of a spectrum $X$ equipped with a trivial $C_2$-action may be described as a homotopy limit
$$X^{tC_2} \simeq \lim_{k} (X \wedge \Sigma P^\infty_{-k}).$$
The AHSS for $\mathit{tmf}^{tC_2}$ arises from the cellular filtration of $\Sigma P^\infty_{-\infty}$ by applying $\mathit{tmf}_*(-)$. The filtration quotients of the cellular filtration $F_i/F_{i+1} \simeq S^i$ are just spheres, so the $E_1$-page of the AHSS converging to $\pi_*(\mathit{tmf}^{tC_2})$ is given by
$$E^{s,t}_1 = \mathit{tmf}_t(S^s) \cong \mathit{tmf}_{t-s}(S^0) \Rightarrow \pi_{t}(\mathit{tmf}^{tC_2}).$$
In other words, the $E_1$-page consists of a copy of $\pi_*(\mathit{tmf})$ in each filtration $s$. A picture of these homotopy groups can be found in \cite[``The homotopy groups of $\mathit{tmf}$ and of its localizations"]{DFHH14}. For readability, we have recreated this image (minus some multiplicative extensions) in Figure \ref{Fig:E00}.

For any integer $r>0$, we write $\nu_2(r)$ for its $2$-adic valuation. We only need to compute finitely many differentials on each page in view of the following lemma (which is a consequence of James periodicity):

\begin{lem}
The $d_r$-differentials in the AHSS are periodic with period $(2^{\nu_2(r)},2^{\nu_2(r)})$. 
\end{lem}

The $d_1$-differentials are given by 
$$d_1(x[s]) = 2x[s-1], \quad s \equiv 1 \mod 2,$$ 
where our notation is that $\alpha[i]$ is the copy of the class $\alpha$ occurring in Atiyah-Hirzebruch filtration $i$. We obtain the $E_2$-page depicted in Figures \ref{Fig:E20}-\ref{Fig:E21}. In Figures \ref{Fig:E20}-\ref{Fig:E57}, squares and bullets both represent $\f_2$. 

We next determine the $E_3$-page. The $d_2$-differentials are given by
$$d_2(x[s]) = \eta x[s-2], \quad s \equiv 1,2 \mod 4.$$
We obtain the $E_3$-page depicted in Figures \ref{Fig:E30}-\ref{Fig:E33}. 

To compute the $E_4$-page, we have $d_3$-differentials 
$$d_3(x[s]) = \langle x,2,\eta \rangle[s-3], \quad s \equiv 3 \mod 4,$$
$$d_3(x[s]) = \langle x, \eta,2 \rangle[s-3], \quad s \equiv 1 \mod 4.$$
To compute these differentials we will use the Massey products from \cite[Table 16]{Isa14} and Toda brackets from \cite[Table 19]{Isa14}. The unit map 
$$S^0 \to \mathit{tmf}$$
is a map of $E_\infty$-ring spectra and therefore the induced map in homotopy groups preserves all higher structure. In particular, we can use Isaksen's computations to compute Toda brackets of classes in $\mathit{tmf}_*$ as long as those classes are in the Hurewicz image for $\mathit{tmf}$. 

\begin{lem}\cite[Table 19]{Isa14}
The following Toda brackets hold in $\pi_*(S^0)$:
$$\eta^2 \in \langle 2, \eta, 2 \rangle, \quad 2 \nu \in \langle \eta, 2, \eta \rangle,  \quad \epsilon \in \langle \nu^2, 2, \eta \rangle.$$
The following Toda brackets hold in $\mathit{tmf}_*$ for $i \geq 1$:
$$c_4^{i-1} 2c_6 \in \langle c_4^i \eta^2, 2, \eta \rangle.$$
\end{lem}

\begin{lem}
For $s \equiv 3 \mod 4$, there are nontrivial $d_3$-differentials
$$d_3(\nu^2[s]) = \epsilon[s-3], \quad d_3(\Delta \eta[s]) = \Delta 2 \nu[s-3], \quad d_3(\nu \Delta^2 \nu^2[s]) = \nu \Delta^2 \epsilon[s-3],$$
$$d_3(\nu \Delta^4 \nu[s]) = \epsilon \Delta^4[s-3], \quad d_3(\kappa \Delta^4 \eta \Delta[s]) = \kappa \Delta^4 2 \nu \Delta[s-3], \quad d_3(\nu \Delta^6 \nu^2[s]) = \nu \Delta^6 \epsilon [s-3].$$
For $s \equiv 1 \mod 4$, there are nontrivial $d_3$-differentials
$$d_3(2 \bar{\kappa}^2[s]) = \eta^2 \bar{\kappa}^2[s-3], \quad d_3(\nu \Delta^5 2[s]) = \nu \Delta^5 \eta^2[s-3].$$
\end{lem}

\begin{proof}
These follow from inspection of \cite[Pages 190-191]{DFHH14}. 
\end{proof}

The $E_4$-page is depicted in Figures \ref{Fig:E40}-\ref{Fig:E43}. Before proceeding to the $E_5$-page, we note that a large portion of the elements in the $E_4$-page survive to the $E_\infty$-page. 

\begin{lem}\label{Lem:J}
Nontrivial elements in the $E_4$-page of the AHSS of the form $c_4^i \eta^j \Delta^k[-N]$ or $2c_6 c_4^i \Delta^k[-N]$ with $i \geq 1$, $1 \leq j \leq 3$, $k \geq 0$, and $N >0$ survive to nontrivial elements in the $E_\infty$-page of the AHSS. 
\end{lem}

\begin{proof}
There are no multiplicative or higher multiplicative relations in $\mathit{tmf}_*$ involving these classes which could produce a $d_r$-differential killing them in the AHSS for $r \geq 4$. 
\end{proof}

Next we compute the $E_5$-page. The $d_4$-differentials are given by
$$d_4(x[s]) = \nu x [s-4], \quad s \equiv 1,2,3,4 \mod 8.$$ 
The $E_5$-page is depicted in Figures \ref{Fig:E50}-\ref{Fig:E57}. 

There are no possible $d_5$-differentials, since these would either correspond to attaching maps in $\pi_4(S^0) = 0$ or a Toda bracket $\langle x, \nu, 2 \rangle$ which is not defined since $2\nu \neq 0$, or a Toda bracket $\langle x, \eta, 2, \eta \rangle$ which is zero for all $x \in \mathit{tmf}_*$. Therefore we have $E_5 = E_6$.

We now calculate the $E_7$-page. Since $\pi_5(S^0) = 0$, the $d_6$-differentials are given by
$$d_6(x[s]) = \langle x, \eta,\nu \rangle[s-6], \quad s \equiv 5,6 \mod 8,$$
$$d_6(x[s]) = \langle x, \nu, \eta \rangle[s-6], \quad s \equiv 1,2 \mod 8.$$
To compute these differentials, we will use the Massey products from \cite[Table 1]{Bau08} and \cite[Table 19]{Isa14}.

\begin{lem}\cite[Table 1]{Bau08}\cite[Table 19]{Isa14}
The following Toda brackets hold in $\mathit{tmf}_*$:
$$\nu^2 \in \langle \eta, \nu, \eta \rangle, \quad \epsilon \in \langle \nu, \eta, \nu \rangle, \quad \epsilon \in \langle 2\nu, \nu, \eta \rangle, \quad 2 \bar{\kappa} \in \langle \kappa \eta, \eta, \nu \rangle.$$
\end{lem}

\begin{lem}
We have the following $d_6$-differentials. For $s \equiv 6 \mod 8$, we have
$$d_6(\kappa \Delta^4 \eta[s]) = 2 \Delta^4 \bar{\kappa}[s-6].$$
For $s \equiv 1 \mod 8$, we have
$$d_6(\kappa \eta[s]) = \kappa \nu^2[s-6], \quad d_6(\nu \Delta^2 2\nu[s]) = \nu \Delta^2 \epsilon[s-6], \quad d_6(\kappa \Delta^4 \eta[s]) = \kappa \Delta^4 \nu^2.$$
\end{lem}

As the $E_7$-page is fairly similar to the $E_5$-page, we leave the construction of its charts as an exercise for the reader.

We conclude with the $E_8$-page. The possible $d_7$-differentials are given by
$$d_7(x[s]) = \langle x, 2,\eta,\nu \rangle[s-7], \quad s \equiv 7 \mod 8,$$
$$d_7(x[s]) = \langle x, \nu, \eta, 2 \rangle[s-7], \quad s \equiv 1 \mod 8.$$
These follow from the Toda bracket $\bar{\kappa} \in \langle \kappa, 2, \eta, \nu \rangle$ in $\mathit{tmf}_*$ \cite[Table 1]{Bau08}.

We then have the following list of $d_7$-differentials.

\begin{lem}
For $s \equiv 7 \mod 8$, there are nontrivial $d_7$-differentials
$$d_7(\kappa[s]) = \bar{\kappa}[s-7], \quad d_7(\eta \Delta \kappa[s]) = \eta \Delta \bar{\kappa}[s-7], \quad d_7(\kappa \Delta^4 \eta[s]) = \eta\Delta^4 \bar{\kappa}[s-7], \quad d_7(\kappa \Delta^4 \kappa[s]) = \kappa \Delta^4 \bar{\kappa}[s-7].$$
\end{lem}

The $E_8$-page is depicted in a range in Appendix \ref{App:AHSS}.

\begin{rem2}
We have only listed differentials on elements in $\pi_{\leq 191}\mathit{tmf}$. The remaining differentials are determined by the rule $d_r(\Delta^8 x) = \Delta^8 d_r(x)$ which follow from the fact that $\Delta^8 \in \pi_{192}(\mathit{tmf})$ is nonzero. 
\end{rem2}

\begin{landscape}

\begin{figure}
\centering
\includegraphics[scale=.75,trim={3cm 15cm 1cm 2cm},clip]{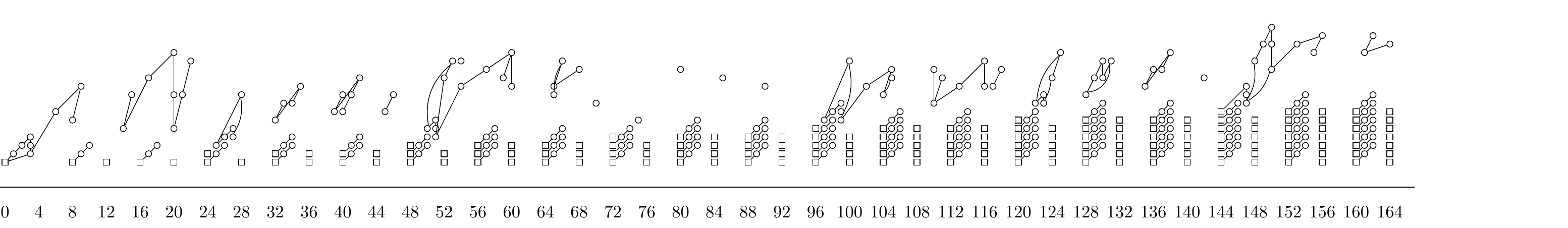}
\caption{The $E_1$-page of the AHSS for $\mathit{tmf}^{tC_2}$ in any filtration $s \in \z$}\label{Fig:E00}
\end{figure}

\begin{figure}
\centering
\includegraphics[scale=.75,trim={3cm 15cm 1cm 2cm},clip]{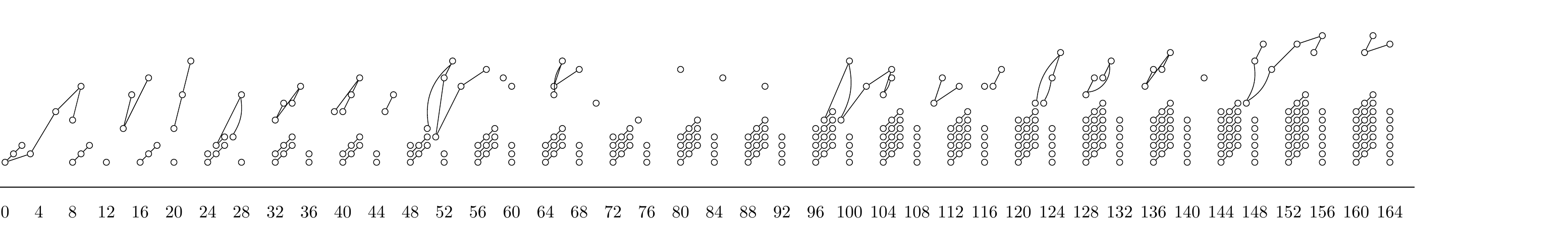}
\caption{The $E_2$-page of the AHSS for $\mathit{tmf}^{tC_2}$ in filtration $s \equiv 0 \mod 2$}\label{Fig:E20}
\end{figure}

\begin{figure}
\centering
\includegraphics[scale=.75,trim={3cm 15cm 1cm 2cm},clip]{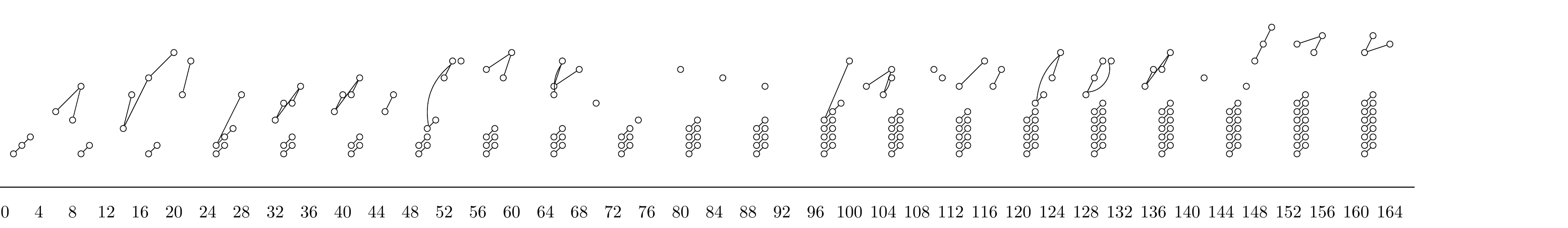}
\caption{The $E_2$-page of the AHSS for $\mathit{tmf}^{tC_2}$ in filtration $s \equiv 1 \mod 2$}\label{Fig:E21}
\end{figure}

\begin{figure}
\centering
\includegraphics[scale=.75,trim={3cm 15cm 1cm 2cm},clip]{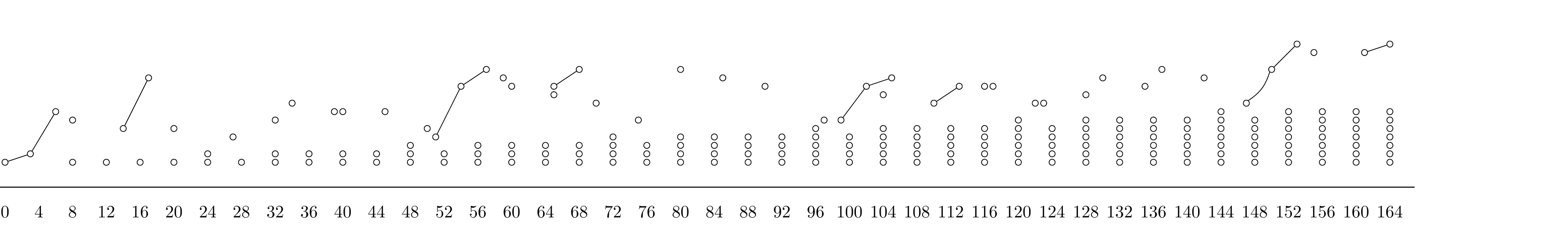}
\caption{The $E_3$-page of the AHSS for $\mathit{tmf}^{tC_2}$ in filtration $s \equiv 0 \mod 4$}\label{Fig:E30}
\end{figure}

\begin{figure}
\centering
\includegraphics[scale=.75,trim={3cm 15cm 1cm 2cm},clip]{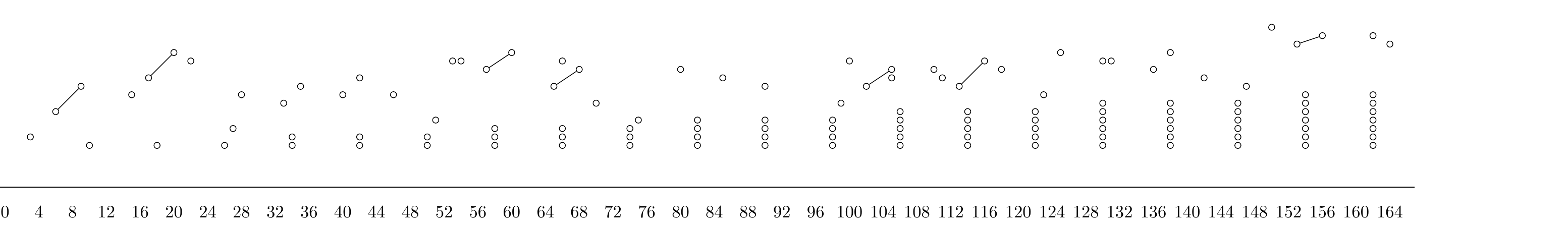}
\caption{The $E_3$-page of the AHSS for $\mathit{tmf}^{tC_2}$ in filtration $s \equiv 1 \mod 4$}\label{Fig:E31}
\end{figure}

\begin{figure}
\centering
\includegraphics[scale=.75,trim={3cm 15cm 1cm 2cm},clip]{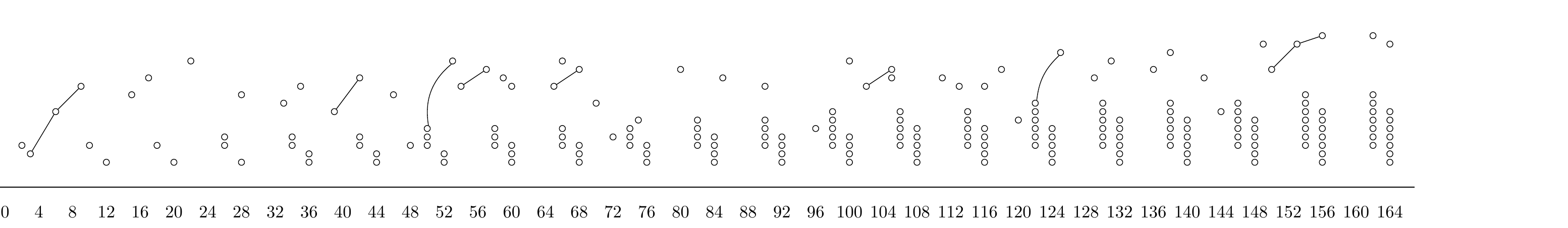}
\caption{The $E_3$-page of the AHSS for $\mathit{tmf}^{tC_2}$ in filtration $s \equiv 2 \mod 4$}\label{Fig:E32}
\end{figure}

\begin{figure}
\centering
\includegraphics[scale=.75,trim={3cm 15cm 1cm 2cm},clip]{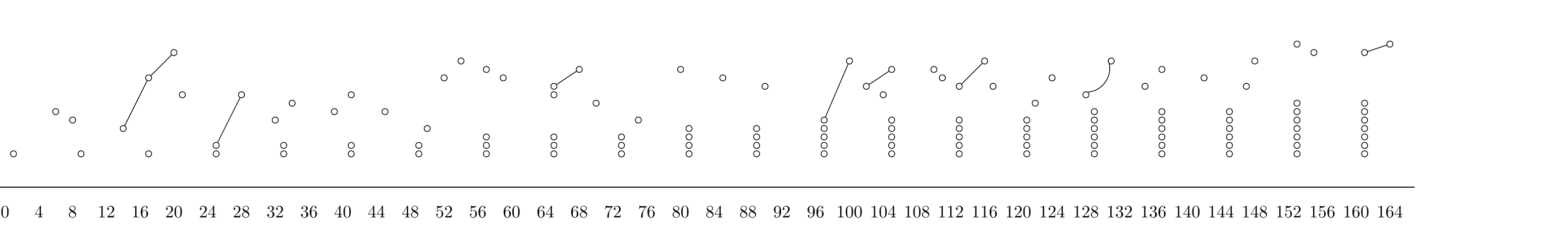}
\caption{The $E_3$-page of the AHSS for $\mathit{tmf}^{tC_2}$ in filtration $s \equiv 3 \mod 4$}\label{Fig:E33}
\end{figure}

\begin{figure}
\centering
\includegraphics[scale=.75,trim={3cm 15cm 1cm 2cm},clip]{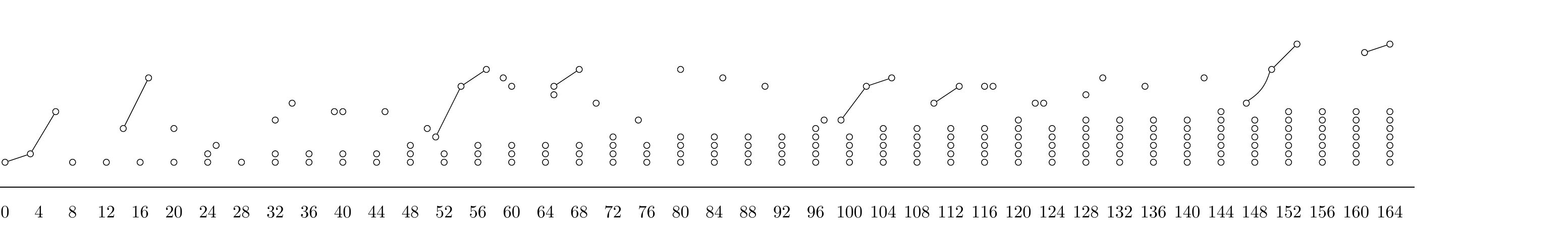}
\caption{The $E_4$-page of the AHSS for $\mathit{tmf}^{tC_2}$ in filtration $s \equiv 0 \mod 4$}\label{Fig:E40}
\end{figure}

\begin{figure}
\centering
\includegraphics[scale=.75,trim={3cm 15cm 1cm 2cm},clip]{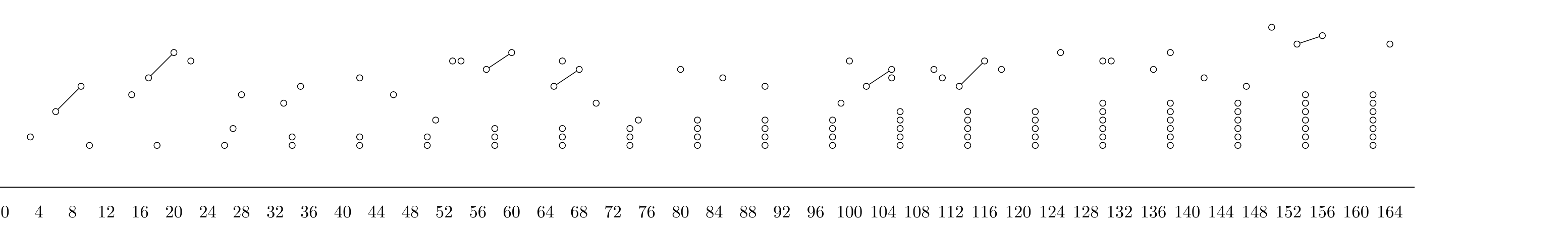}
\caption{The $E_4$-page of the AHSS for $\mathit{tmf}^{tC_2}$ in filtration $s \equiv 1 \mod 4$}\label{Fig:E41}
\end{figure}

\begin{figure}
\centering
\includegraphics[scale=.75,trim={3cm 15cm 1cm 2cm},clip]{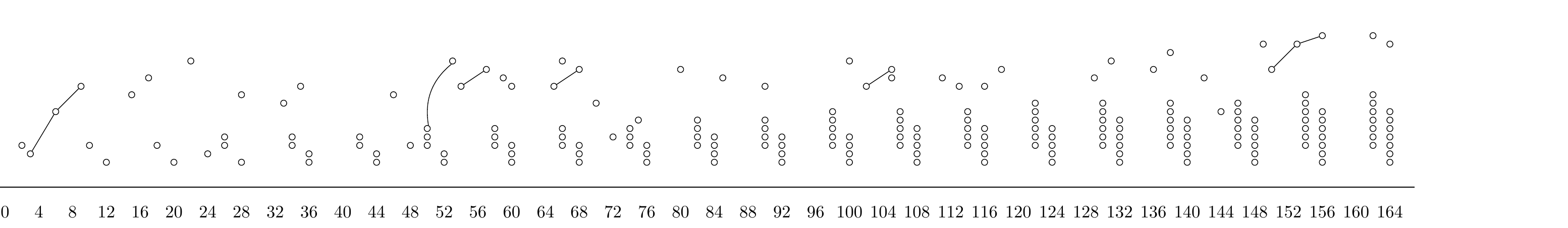}
\caption{The $E_4$-page of the AHSS for $\mathit{tmf}^{tC_2}$ in filtration $s \equiv 2 \mod 4$}\label{Fig:E42}
\end{figure}

\begin{figure}
\centering
\includegraphics[scale=.75,trim={3cm 15cm 1cm 2cm},clip]{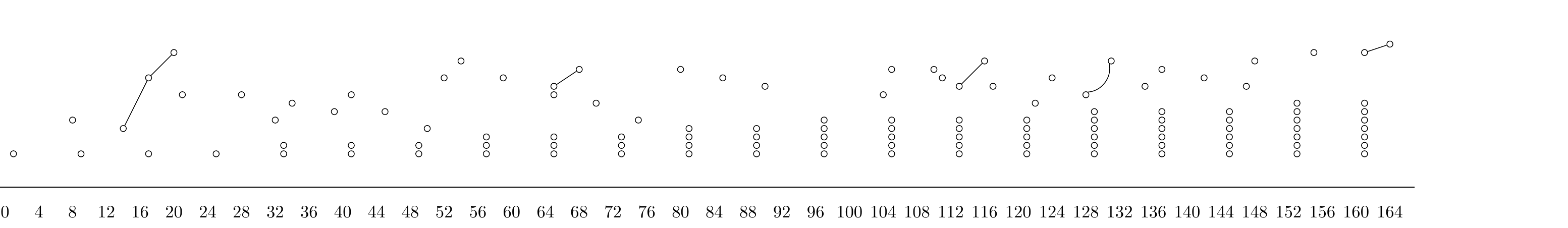}
\caption{The $E_4$-page of the AHSS for $\mathit{tmf}^{tC_2}$ in filtration $s \equiv 3 \mod 4$}\label{Fig:E43}
\end{figure}

\begin{figure}
\centering
\includegraphics[scale=.75,trim={3cm 15cm 1cm 2cm},clip]{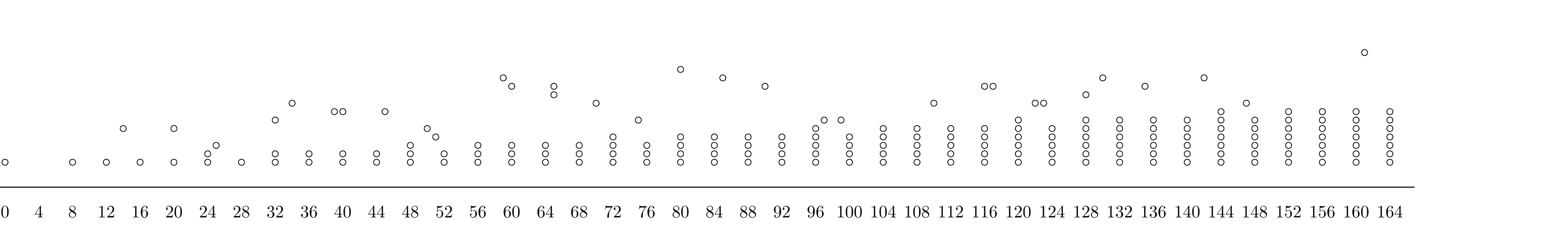}
\caption{The $E_5$-page of the AHSS for $\mathit{tmf}^{tC_2}$ in filtration $s \equiv 0 \mod 8$}\label{Fig:E50}
\end{figure}

\begin{figure}
\centering
\includegraphics[scale=.75,trim={3cm 15cm 1cm 2cm},clip]{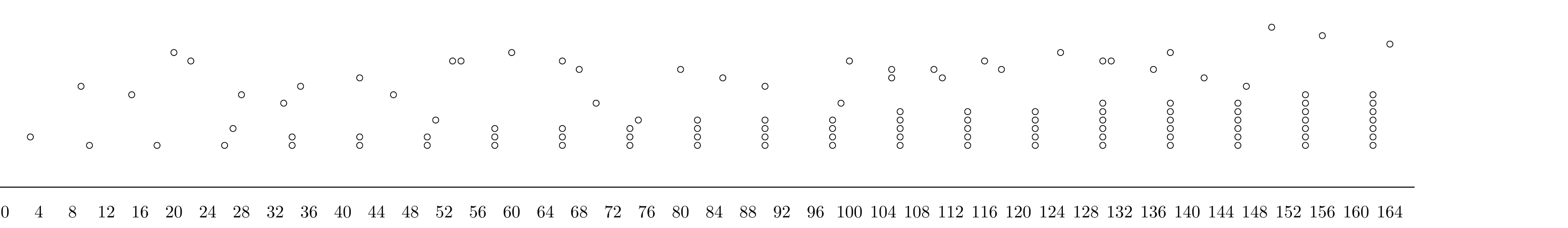}
\caption{The $E_5$-page of the AHSS for $\mathit{tmf}^{tC_2}$ in filtration $s \equiv 1 \mod 8$}\label{Fig:E51}
\end{figure}

\begin{figure}
\centering
\includegraphics[scale=.75,trim={3cm 15cm 1cm 2cm},clip]{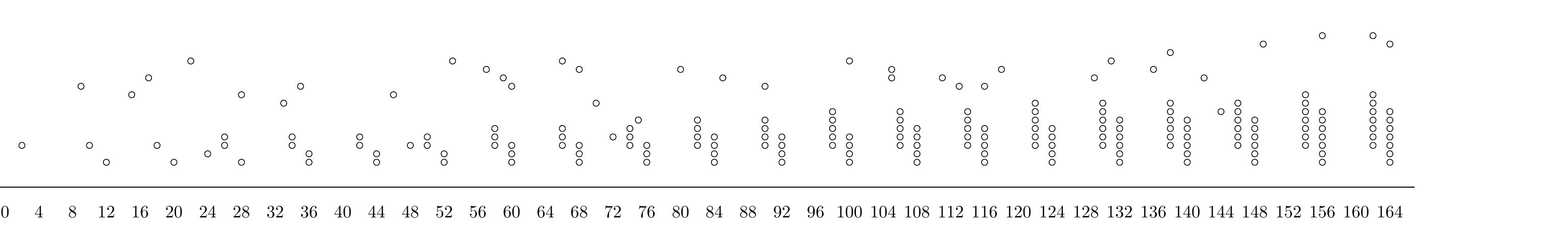}
\caption{The $E_5$-page of the AHSS for $\mathit{tmf}^{tC_2}$ in filtration $s \equiv 2\mod 8$}\label{Fig:E52}
\end{figure}

\begin{figure}
\centering
\includegraphics[scale=.75,trim={3cm 15cm 1cm 2cm},clip]{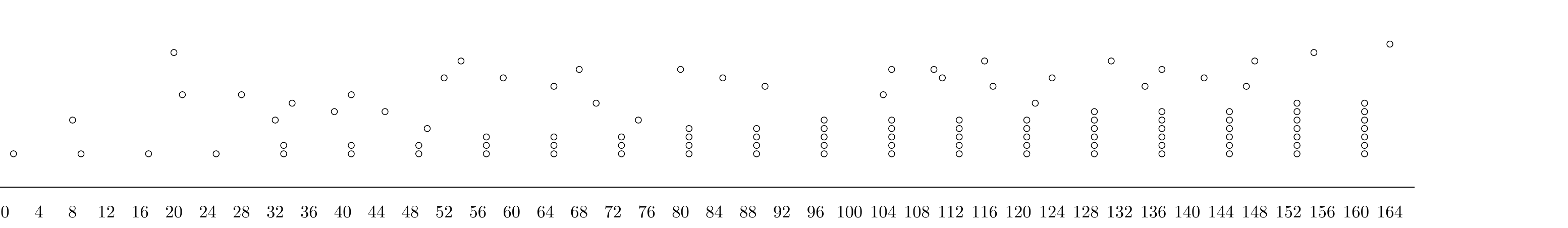}
\caption{The $E_5$-page of the AHSS for $\mathit{tmf}^{tC_2}$ in filtration $s \equiv 3 \mod 8$}\label{Fig:E53}
\end{figure}

\begin{figure}
\centering
\includegraphics[scale=.75,trim={3cm 15cm 1cm 2cm},clip]{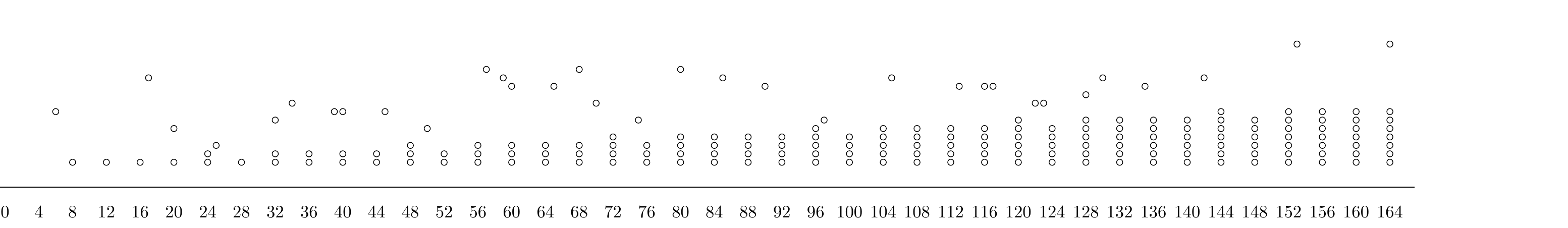}
\caption{The $E_5$-page of the AHSS for $\mathit{tmf}^{tC_2}$ in filtration $s \equiv 4 \mod 8$}\label{Fig:E54}
\end{figure}

\begin{figure}
\centering
\includegraphics[scale=.75,trim={3cm 15cm 1cm 2cm},clip]{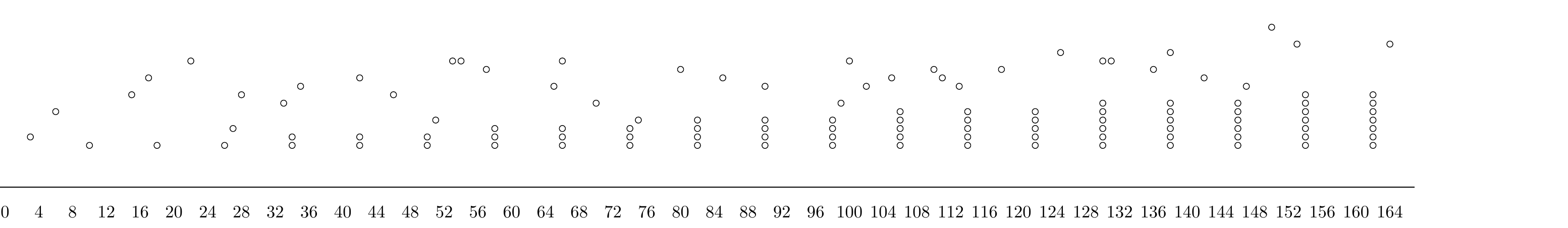}
\caption{The $E_5$-page of the AHSS for $\mathit{tmf}^{tC_2}$ in filtration $s \equiv 5 \mod 8$}\label{Fig:E55}
\end{figure}

\begin{figure}
\centering
\includegraphics[scale=.75,trim={3cm 15cm 1cm 2cm},clip]{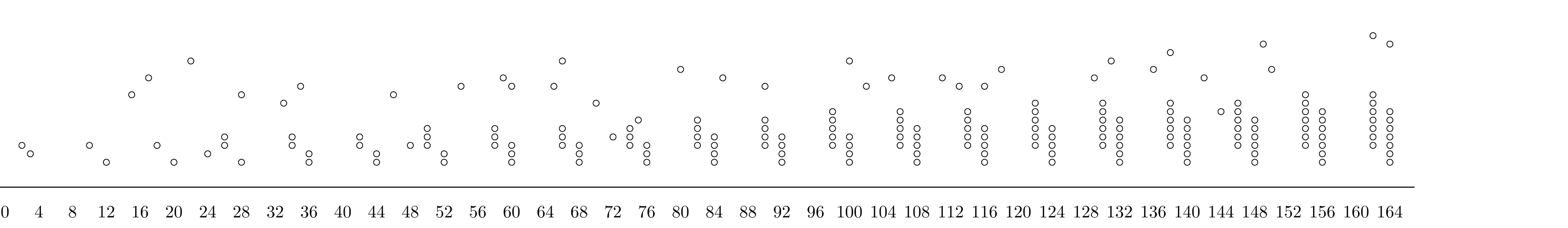}
\caption{The $E_5$-page of the AHSS for $\mathit{tmf}^{tC_2}$ in filtration $s \equiv 6 \mod 8$}\label{Fig:E56}
\end{figure}

\begin{figure}
\centering
\includegraphics[scale=.75,trim={3cm 15cm 1cm 2cm},clip]{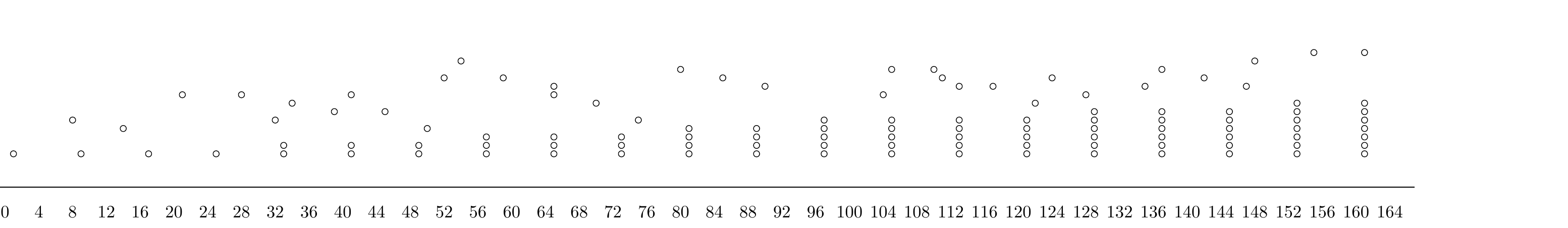}
\caption{The $E_5$-page of the AHSS for $\mathit{tmf}^{tC_2}$ in filtration $s \equiv 7 \mod 8$}\label{Fig:E57}
\end{figure}
\end{landscape}

\section{Algebraic $E$-based Mahowald invariants via the Koszul spectral sequence}\label{SectionAlg}
In this section, we recall Davis and Mahowald's computation of the algebraic $\mathit{tmf}$-based Mahowald invariants which will serve as our starting point for calculating $\mathit{tmf}$-based Mahowald invariants in Section \ref{SectionFiltered}. Let $M$ be an $A(n)$-module. Define $R_n := P^\otimes(\xi_1^{2^n}, \xi_2^{2^{n-1}},\ldots,\xi^2_n,\xi_{n+1})$ to be the tensor-algebra on the above generators, and let $R^\sigma_n$ denote the submodule consisting of monomials of length $n$. The \emph{Koszul spectral sequence} \cite[Theorem 2.8]{MS87} has the form
$$E^{\sigma,s,t}_1 = Ext^{s-\sigma,t}_{A(n-1)}(R^\sigma_n \otimes M) \Rightarrow Ext^{s,t}_{A(n)}(M).$$
This spectral sequence was used by Davis and Mahowald \cite{DM82} to compute $Ext_{A(2)}(H^*(RP^\infty_N))$ for all $N \in \z$ and by Mahowald and Shick \cite{MS87} to define the chromatic filtration of the $E_2$-page of the ASS. It also appears as the ``Davis-Mahowald spectral sequence" in recent work of Shick \cite{Shi19}. 

\subsection{The Koszul spectral sequence for $Ext_{A(1)}^{**}(H^*(\Sigma P^\infty_N))$}
In this subsection, we demonstrate the Koszul spectral sequence in a simple case. We include many details to give the reader an idea of how to understand the computations in \cite{DM82}, which we will cite but not reprove in the sequel. The first Koszul spectral sequence we consider has the form
$$E^{\sigma,s,t}_1 = Ext^{s-\sigma,t}_{A(0)}(P^\otimes(\xi^2_1,\xi_2)^\sigma) \Rightarrow Ext^{s,t}_{A(1)}(\f_2).$$

We begin by calculating the $E_1$-page. Let $R^\sigma_1 = P(\xi^2_1,\xi_2)^\sigma$. 

\begin{lem}\label{dm1}
When $\sigma$ is even, there is an isomorphism
$$Ext^{*-\sigma,*}_{A(0)}(R^\sigma_1) =  \left(\bigoplus_{i = 0}^{\sigma-1} \f_2[2i] \right) \oplus \z[\sigma]$$
where the functor $(-)[k]$ shifts elements in bidegree $(s,t)$ into bidegree $(s,t+k)$. When $\sigma$ is odd, there is an isomorphism 
$$Ext^{*-\sigma, *}_{A(0)}(R^\sigma_1) = \bigoplus_{i=0}^{\sigma}\f_2[2i].$$
\end{lem}

\begin{proof}
The $A(0)$-module $R^1_1$ consists of classes $[\xi^2_1]$ and $[\xi_2]$. The $A(0)_*$-coaction on $\xi_2$ is restricted from the $A_*$-coaction, so we have
$$\psi(\xi_2) = \xi_1 \otimes \xi^2_1 + 1 \otimes \xi_2,$$
from which we conclude $Sq^1(\xi_2) = \xi_1^2.$ Therefore 
$$R^1_1 \cong \Sigma^2 A(0) \cong H^*(\Sigma^2V(0)),$$
and $Ext_{A(0)}(R^1_1)$ is just a shifted copy of $\f_2$. 

The $A(0)$-module $R^2_1$ consists of classes $[\xi^2_1 \otimes \xi^2_1]$, $[\xi^2_1 \otimes \xi_2]$, and $[\xi_2 \otimes \xi_2]$. By the same argument as above, we see that $Sq^1([\xi^2_1 \otimes \xi_2]) = [\xi^2_1 \otimes \xi^2_1]$, and $Sq^1[\xi_2 \otimes \xi_2]=0$. Therefore we conclude 
$$R^2_1 \cong \Sigma^4 A(0) \oplus \Sigma^6 \f_2 \cong H^*(\Sigma^4 V(0) \wedge S^6),$$
and $Ext_{A(0)}(R^2_1)$ is a shifted copy of $\f_2$ and a shifted copy of $\z$. 

More generally, the elements of $R_1^\sigma$ are in bijective correspondence with monotone increasing strings $(1, 1, \ldots, 2, 2)$ of length $\sigma$. Elements with $2i$ entries of `2' are related by $Sq^1$ to elements with $2i+1$ entries of `2'. The result then follows easily by induction on $\sigma$. 
\end{proof}

We now calculate differentials. The $d_1$-differentials in the Koszul spectral sequence are obtained by dualizing the differential in the Koszul resolution (see e.g. \cite[Sec. 2]{MS87}). Using this fact, or by comparing to other calculations of $Ext_{A(1)}^{**}(\f_2)$, we obtain the pattern of differentials depicted in Figure \ref{Fig:KoszulSSF2}. 

\begin{figure}
\centering
\includegraphics[trim={2in 7.5in 3in 1.5in},clip,scale=.8]{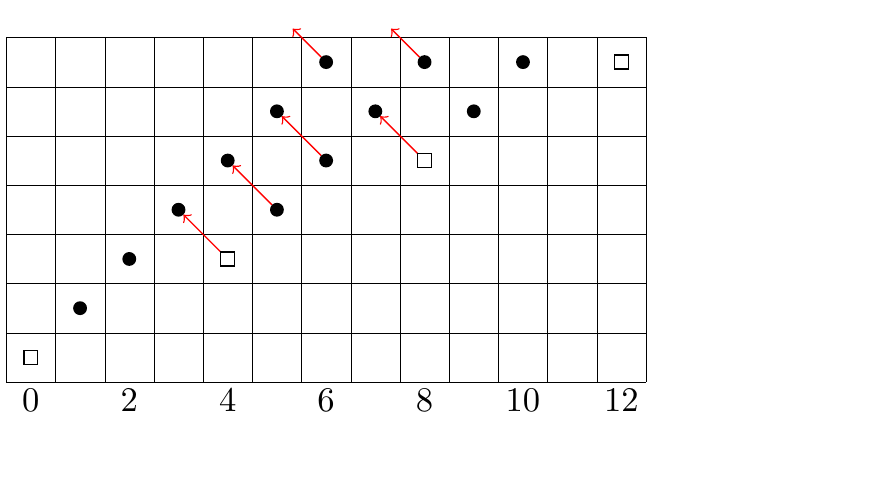}
\caption{The Koszul spectral sequence for $Ext_{A(1)}^{**}(\f_2)$. A $\square$ represents $\z$ and a $\bullet$ represents $\f_2$. }\label{Fig:KoszulSSF2}
\end{figure}

We can use the Koszul spectral sequence to compute $Ext_{A(1)}(H^*(\Sigma P^\infty_N))$ for all $N \in \z$. In this case, it takes the form
$$E_1^{\sigma,s,t} = Ext^{s,t}_{A(0)}(R_1^\sigma \otimes H^*(\Sigma P^\infty_N)) \Rightarrow Ext^{\sigma+s,t}_{A(1)}(H^*(\Sigma P^\infty_N)).$$ 
Recall that as an $A(0)$-module, $H^*(\Sigma P^\infty_N)$ is $2$-periodic with respect to $N$. In particular, we have isomorphisms of $A(0)$-modules
$$H^*(\Sigma P^\infty_{2N-1}) \cong \bigoplus_{i \geq N} \Sigma^{2i} H^*(V(0)),$$
$$H^*(\Sigma P^\infty_{2N}) \cong \Sigma^{2N+1} H^*(pt) \oplus \bigoplus_{i \geq N} \Sigma^{2i+2} H^*(V(0)).$$
The first splitting shows that in order to understand the Koszul spectral sequence for $H^*(\Sigma P^\infty_{2N-1})$, it suffices to analyze the Koszul spectral sequence for $V(0)$. The second splitting shows that it suffices to understand the spectral sequence for $H^*(V(0))$ and $H^*(S^0)$. 

For $H^*(S^0)$, the Koszul spectral sequence was already analyzed. For $H^*(V(0))$, we recall that $H^*(V(0)) \cong A(0)$ so we may apply change-of-rings to see that
$$Ext^{**}_{A(0)}(R_1^\sigma \otimes_{\f_2} H^*(V(0))) \cong Ext^{**}_{\f_2}(R_1^\sigma).$$
The differentials can again be calculated from the definition or from comparison with existing calculations of $Ext_{A(1)}^{**}(H^*(V(0)))$; we leave the details to the reader. 

Altogether, this gives us sufficient information to calculate $Ext^{**}_{A(1)}(H^*(\Sigma P^\infty_{N}))$ for all $N \in \z$. We include a picture of $Ext_{A(1)}^{**}(H^*(\Sigma P^\infty_{-9}))$ in Figure \ref{Fig:KoszulSSP9} as an example. 

\begin{figure}
\centering
\includegraphics[trim={1.5in 7in 3in 1.5in},clip,scale=.8]{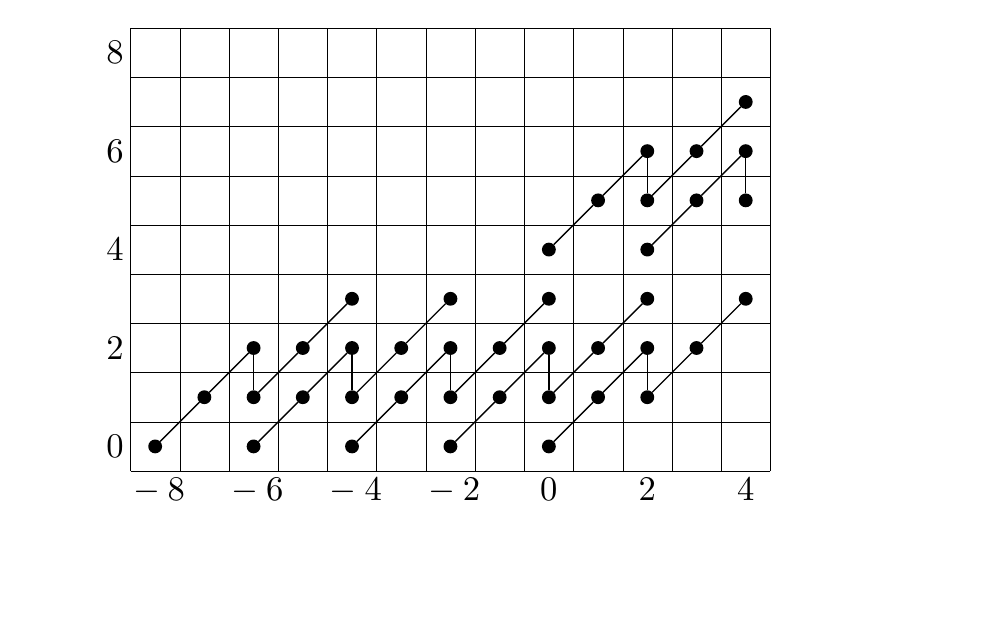}
\caption{The $E_\infty$-page of the Koszul spectral sequence for $Ext_{A(1)}^{**}(H^*(\Sigma P^\infty_{-9}))$. A $\bullet$ represents $\f_2$.}\label{Fig:KoszulSSP9}
\end{figure}

\subsection{Algebraic $E$-based Mahowald invariants}
Let $\alpha \in Ext^{s,t}_{A(n)}(H^*(\Sigma P^\infty_{-\infty}))$. Applying $Ext_{A(n)}^{**}(H^*(-))$ to Diagram \ref{Eqn:MI} produces a diagram
\[
\begin{tikzcd}
Ext^{s,t}_{A(n)}(\f_2) \arrow{d} \arrow[dashed,r,"M^{alg}_E(-)"] & Ext^{s,t+N-1}_{A(n)}(\f_2) \arrow{d} \\
Ext^{s,t}_{A(n)}(H^*(\Sigma P^\infty_{-\infty})) \arrow{r}{\nu_N} &Ext^{s,t}_{A(n)}(H^*(P^\infty_{-N})).
\end{tikzcd}
\]
where $N>0$ is minimal such that the left-hand composite is nontrivial (compare with \cite[Def. 3.3]{Beh06}). If $E$ is a spectrum with cohomology $A\modmod A(n)$, then the coset of lifts $\gamma \in Ext^{s,t-N-1}_{A(n)}(\f_2)$ of the element $\nu_N \circ f(\alpha)$ is the \emph{algebraic $E$-based Mahowald invariant of $\alpha$}. Recall that $H^*(bo) \cong A\modmod A(1)$ and $H^*(\mathit{tmf}) \cong A\modmod A(2)$, so we may define the algebraic $bo$-based and algebraic $\mathit{tmf}$-based Mahowald invariants. 

Let $n=1,2$. Elementary calculations with the Steenrod algebra and the $A$-module structure of $H^*(\Sigma P^\infty_{-\infty})$ show that as an $A(n)$-module, $H^*(\Sigma P^\infty_{-\infty})$ splits as
$$H^*(\Sigma P^\infty_{-\infty}) \cong \bigoplus_{i \in \z} \Sigma^{i 2^{n+1}} A(n)\modmod A(n-1).$$
By change-of-rings, we then have
$$Ext^{s,t}_{A(n)}(H^*(\Sigma P^\infty_{-\infty})) \cong \bigoplus_{i \in \z} \Sigma^{i 2^{n+1}} Ext_{A(n-1)}(\f_2)$$
so $M_{alg}^E(\alpha)$ is defined for any $\alpha \in Ext^{**}_{A(n-1)}(\f_2)$. 

The analysis in the previous section determined the bottom left term in the diagram for all $N\in \z$ when $n=1$. We can therefore compute the algebraic $bo$-based Mahowald invariants by inspection of the relevant $Ext$-groups and analysis of the order of $2$-torsion in $Ext^{s,s}_{A(1)}(H^*(P^\infty_{-2n-1}))$. Since these calculations are similar to the calculations of the homotopical $bo$-based Mahowald invariants of $2^i$ which were computed by Mahowald and Ravenel in \cite{MR93}, we leave the details to the reader. The result is given in the following theorem, where $h_1$ is the generator of $Ext^{1,2}_{A(1)}(\f_2) \cong \f_2$, $\alpha$ is the generator of $Ext^{3,7}_{A(1)}(\f_2) \cong \f_2$, and $\beta$ is the generator of $Ext^{4,12}_{A(1)}(\f_2) \cong \f_2$.  

\begin{thm}
The algebraic $bo$-based Mahowald invariant is determined by the relation $M^{alg}_{bo}(h^4_0 x) = \beta M^{alg}_{bo}(x)$ and the table 
\begin{center}
\begin{tabular}{ |c|c|} 
 \hline
$x$ & $M^{alg}_{bo}(x)$ \\
\hline
$1$ & $1$ \\
$h_0$ & $h_1$ \\
$h^2_0$ & $h^2_1$ \\
$h^3_0$ & $\alpha$ \\
 \hline
\end{tabular}
\end{center}
\end{thm}

\begin{rem2}
These algebraic $bo$-based Mahowald invariants agree with the (non-algebraic) $bo$-based Mahowald invariants $M_{bo}(2^i)$ after replacing $x \in Ext_{A(1)}$ by the element which it detects in the ASS for $bo_*$. The $bo$-based Mahowald invariants can be lifted to spherical Mahowald invariants using the $bo$-filtered Mahowald invariant \cite[Section 7]{Beh06}. 
\end{rem2}

\subsection{Algebraic $\mathit{tmf}$-based Mahowald invariants}
We now summarize David and Mahowald's analysis of the Koszul spectral sequence for $Ext_{A(2)}^{**}(\f_2)$ and $Ext_{A(2)}^{**}(H^*(P^\infty_N))$. The first of these has the form
$$E^{\sigma,s,t}_1 = Ext_{A(1)}(P(\xi^4_1,\xi^2_2,\xi_3)^\sigma) \Rightarrow Ext_{A(2)}(\f_2).$$
In \cite[Proposition 5.4]{DM82}, they compute the $E_1$-term of this spectral sequence by computing $Ext_{A(1)}(R^\sigma_2)$ using the same ideas as in our proof of Lemma \ref{dm1}. They then show that $d_1$ is $h_2$- and $v^8_2$-linear, and arrive at a closed formula for the $d_1$-differential \cite[Theorem 5.6]{DM82}. Finally, they show that $E_2 = E_\infty$. The chart is depicted in \cite[Theorem 2.6]{DM82}. One may compare their calculations to those of Iwai-Shimada \cite{IS67} or the machine calculations depicted in \cite[Figure 6.1]{BHHM08}. 

After understanding $\f_2$, they turn to $H^*(\Sigma P^\infty_N)$. Recall that as an $A(1)$-module, $H^*(\Sigma P^\infty_N)$ decomposed into infinitely many suspensions of $A(1)\modmod A(0)$, and as an $A(2)$-module, $H^*(\Sigma P^\infty_N)$ decomposed into infinitely many suspensions of $A(2)\modmod A(1)$. In fact, $H^*(\Sigma P^\infty_{N})$ may be decomposed as an $A(2)$-module in several different ways. 

\begin{exm}
Bailey and Ricka define $L_0$ to be the spectrum $(S^1 \cup_2 e^2 \cup_\eta e^4 \cup_\nu e^8)_+$, and they construct a cofiber sequence of $\mathit{tmf}$-modules
$$\Sigma^{-8k} \mathit{tmf} \wedge P^\infty_{-1} \to \mathit{tmf} \wedge P^\infty_{-1-8k} \to \Sigma^{-8k} \mathit{tmf} \wedge L_0.$$
This cofiber sequence plays a fundamental role in their proof of the spectral splitting of $\mathit{tmf}^{tC_2}$ \cite[Theorem 1.1]{BR19}. 
\end{exm}

Davis and Mahowald use a different splitting than Baily and Ricka which we outline now. They begin by defining $M_i$ (for $i=0,1,3,4,7$) to be the smallest $A(2)$-module in which $Sq^ig_0 \neq 0$, where $g_0$ is a class in degree zero. For example, $H^*(L_0) \cong \Sigma M_7 \oplus \f_2$. They then use the Koszul spectral sequence to compute $Ext_{A(2)}(M_i)$. These $Ext$-groups are presented in \cite[Theorems 2.6-2.9]{DM82} and can also be computed using $Ext$-calculators of Bruner \cite{Bru93} or Perry \cite{Per13}. 

Given these $Ext$-groups, one can then compute $Ext_{A(2)}(H^*(P^\infty_N))$ for $N$ odd using exact sequences relating $P^\infty_N$ to $A(2)$-modules whose cohomology is known. For example, the exact sequence 
$$0 \to H^*(P^\infty_1) \to Q \to \Sigma^{-9} M_7 \oplus \Sigma^{-1} \f_2 \to 0$$
is used to compute $Ext_{A(2)}(H^*(P^\infty_1))$. Here, $Q$ is the quotient of $H^*(P^\infty_{-\infty})$ by the $A(2)$-module generated by classes in degree less than $-9$. We note that $Q$ appears in the calculation of $Ext_{A(2)}(H^*(P^\infty_N))$ for all $N$ odd. Once one has these calculations for $N$ odd, one can compute $Ext_{A(2)}(H^*(P^\infty_N))$ for $N$ even using the exact sequences
$$0 \to H^*(P^\infty_N) \to H^*(P^\infty_{N-1}) \to \Sigma^{N-1} \f_2 \to 0 \quad \text{and} \quad 0 \to H^*(P^\infty_{N+1}) \to H^*(P^\infty_N) \to \Sigma^{N} \f_2 \to 0.$$
By the above discussion, we have 
$$Ext_{A(2)}(Q) \cong \bigoplus_{i \geq -1} Ext_{A(1)}(\Sigma^{8i-1} \f_2).$$
Let $f : Ext_{A(1)}(\f_2) \to Ext_{A(2)}(Q)$ denote the obvious inclusion. In computing $Ext_{A(2)}(H^*(P^\infty_N))$, one completely determines the image of $Ext_{A(1)}(\f_2)$ inside of $Ext_{A(2)}(H^*(P^\infty_N))$, including the minimal $N<0$ where the image of a class $\alpha \in Ext_{A(1)}(\f_2)$ is first nontrivial under $f_*$. This leads to the following theorem.

\begin{thm}\cite[Theorem 3.5]{DM82}\label{atmf1}
The algebraic $\mathit{tmf}$-based Mahowald invariant is determined by the relations $M^{alg}_{\mathit{tmf}}(h^4_0 x) = v^4_1 M^{alg}_{\mathit{tmf}}(x)$ and $M^{alg}_{\mathit{tmf}}(\beta^2 x) = v^8_2 M^{alg}_{\mathit{tmf}}(x)$ and
\begin{center}
\begin{tabular}{ |c|c|} 
 \hline
$x$ & $M^{alg}_{\mathit{tmf}}(x)$ \\
\hline
$1$ & $1$ \\
$h_0$ & $h_1$ \\
$h_0^2$ & $h^2_1$ \\
$h^3_0$ & $h^3_1$ \\
$h_1$ & $h_2$ \\
\hline
\end{tabular}
\begin{tabular}{ |c|c|}
 \hline
$x$ & $M^{alg}_{\mathit{tmf}}(x)$ \\
\hline
$h_1^2$ & $h^2_2$ \\
$\alpha$ & $a$ \\
$\alpha h_0$ & $d$ \\
$\alpha h^2_0$ & $dh_1$ \\
$\alpha h^3_0$ & $eh^2_0$ \\
\hline
\end{tabular}
\begin{tabular}{ |c|c|}
 \hline
$x$ & $M^{alg}_{\mathit{tmf}}(x)$ \\
\hline 
$\beta$ & $g$ \\
$\beta h_0$ & $gh_1$ \\
$\beta h^2_0$ & $a^2$ \\
$\beta h^3_0$ & $ad$ \\
$\beta h_1$ & $v^4_2 h_1$ \\
 \hline
\end{tabular}
\begin{tabular}{ |c|c|}
 \hline
$x$ & $M^{alg}_{\mathit{tmf}}(x)$ \\
\hline
$\beta h^2_1$ & $v^4_2 h^2_2$ \\
$\beta \alpha$ & $v^4_2 c$\\
$\alpha \beta h_0$ & $dg$ \\
$\alpha \beta h^2_0$ & $a^3$ \\
$\alpha \beta h^3_0$ & $v^4_1 v^4_2 h^2_2$ \\
\hline
\end{tabular}
\end{center}
\end{thm}

\section{$\mathit{tmf}$-based Mahowald invariants}\label{SectionFiltered}
We now apply our partial computation of the AHSS for $\pi_*(\mathit{tmf}^{tC_2})$ to compute the $\mathit{tmf}$-based Mahowald invariants in Theorem \ref{tmfbasedMI}. In Section \ref{Sec:Prelim}, we recall the definition of the $E$-based Mahowald invariant and the filtered Mahowald invariant. We then define the $H$-filtered $\mathit{tmf}$-based Mahowald invariant and record several results needed for our calculations. To help the reader navigate our calculations, we provide a pseudo-algorithm for calculating $\mathit{tmf}$-based Mahowald invariants in Section \ref{Sec:Pseudoalgorithm}. In Section \ref{Sec:v1perMI}, we calculate several $v_1$-periodic $\mathit{tmf}$-based Mahowald invariants which allow us to eliminate many classes in high Adams filtration from consideration in our $v_2$-periodic calculations. In Section \ref{Sec:v1i}, we calculate the  $\mathit{tmf}$-based Mahowald invariants $M_{\mathit{tmf}}(M_{bo}(2^i))$ for all $i \geq 1$. These results are assembled in Theorem \ref{tmfbasedMI}. 

\subsection{Preliminaries}\label{Sec:Prelim}

We start by recalling the $E$-based Mahowald invariant and the filtered Mahowald invariant. We then combine these to define the filtered $\mathit{tmf}$-based Mahowald invariant. 

\subsubsection{$E$-based Mahowald invariants}
We begin with $E$-based Mahowald invariants. 

\begin{defin}
Let $\alpha \in \pi_t(E^{tC_2})$. The \emph{$E$-based Mahowald invariant} of $\alpha$ is the coset of completions of the diagram
\[
\begin{tikzcd}
S^t \arrow{d}{\alpha} \arrow[dashed,rr,"M_{E}(\alpha)"] & & \Sigma^{-N+1} E \arrow{d} \\
E^{tC_2} \arrow{r}{\simeq} & \underset{\leftarrow}{\lim} \Sigma P^\infty_{-N} \wedge E \arrow{r} &  \Sigma P^\infty_{-N} \wedge E
\end{tikzcd}
\]
where $N>0$ is minimal so that the left-hand composite from $S^t$ to $\Sigma P^\infty_{-N} \wedge E$ is nontrivial. 
\end{defin}

\begin{rem2}
The $S^0$-based Mahowald invariant is the ordinary Mahowald invariant defined in the Introduction. We will sometimes refer to this as the \emph{spherical} Mahowald invariant in the sequel. 
\end{rem2}

\begin{rem2}
We can enlarge the diagram above to define $M_{\mathit{tmf}}(M_{bo}(\alpha))$ as the coset of dashed arrows:
\[
\begin{tikzcd}
bo^{tC_2} \arrow{d} & S^t \arrow{l}{\alpha} \arrow{d}{M_{bo}(\alpha)} \arrow[r,dashed,"M_{\mathit{tmf}}(M_{bo}(\alpha))"] & \Sigma^{-N-M+2} \mathit{tmf} \arrow{dd} \\
\Sigma P^\infty_{-N} & \Sigma^{-N+1} bo \arrow{l} \arrow{d} & \\
& \Sigma^{-N+1} \mathit{tmf}^{tC_2} \arrow{r} & \Sigma^{-N+2}P^\infty_{-M} \wedge \mathit{tmf}
\end{tikzcd}
\]
\end{rem2}

As we were unable to completely calculate the AHSS for $\mathit{tmf}^{tC_2}$, we cannot compute $\mathit{tmf}$-based Mahowald invariants immediately from the definition since we cannot determine the minimal $N$ for which any element $\alpha \in \pi_*(\mathit{tmf}^{tC_2})$ is detected. Instead, we will compute the $\mathit{tmf}$-based Mahowald invariant using the filtered Mahowald invariant machinery set up by Behrens in \cite{Beh06}. 

\subsubsection{Filtered Mahowald invariants}
The filtered Mahowald invariant is a sequence of approximations which arise from combining the Adams filtration of $\pi_*(S^0)$ and the Atiyah-Hirzebruch filtration of $\Sigma RP^\infty_{-\infty}$. We will not give a full review of this machinery here; we refer the reader to \cite[Sec. 5]{Beh07} for an introduction and rigorous definition.

\begin{defin}
Let $E$ be a spectrum for which the $E$-based ASS for $S^0$ converges and let $\alpha \in \pi_t(S^0)$. Then the $k$-th $E$-filtered Mahowald invariant of $\alpha$ is defined in \cite[Def. 5.1]{Beh07} and is denoted by 
$$M^{[k]}(\alpha) \in E_1^{k,*}(S^0),$$
where the right-hand side is the $E_1$-page of the $E$-based ASS for $S^0$. We will suppress $E$ from the notation since we will always take $E = H\f_2$ in this paper. 
\end{defin}

\begin{exm}
When $k=1$ and $E=H\f_p$, this recovers the \emph{algebraic Mahowald invariant} studied by Mahowald and Shick in \cite{MS87} and Shick in \cite{Shi87}. The algebraic Mahowald invariants of many low-dimensional classes were computed by Bruner in \cite{Bru98} and serve as a starting point for many of Behrens' computations.
\end{exm}

\begin{rem2}
There is some indeterminacy in the definition of the filtered Mahowald invariant since one can vary the amount of the Atiyah-Hirzebruch filtration considered at each Adams filtration; the situation is depicted graphically in \cite[Fig. 2]{Beh06}. This indeterminacy will be relevant in some of the calculations below. 
\end{rem2}

The filtered Mahowald invariant was originally defined to compute spherical Mahowald invariants, so we must modify it to compute $\mathit{tmf}$-based Mahowald invariants. 

\begin{defin}
Let $E$ be a spectrum for which the $E$-based ASS for $\mathit{tmf}$ converges and let $\alpha \in \pi_t(bo)$. Then the $k$-th $E$-filtered $\mathit{tmf}$-based Mahowald invariant of $\alpha$ is defined by replacing $\pi_*(-)$ by $\mathit{tmf}_*(-)$ in \cite[Def. 5.1]{Beh07}. We will denote it by 
$$M^{[k]}_{\mathit{tmf}}(\alpha) \in E_1^{k,*}(\mathit{tmf})$$
where the right-hand side is the $E_1$-page of the $E$-based ASS for $\mathit{tmf}$. Again, we will suppress $E$ from the notation since we will always take $E = H\f_2$. 
\end{defin}

In order to compute $M^{[k+1]}_{\mathit{tmf}}(\alpha)$ from $M^{[k]}_{\mathit{tmf}}(\alpha)$, we will need the following specializations of \cite[Thm. 6.4]{Beh07} and \cite[Thm. 6.5]{Beh07}. We refer the reader to \cite[Sec. 4]{Beh06} for the definition of the $K$-Toda bracket $\langle K \rangle(\beta)$ (where $K$ is a finite complex and $\beta \in \pi_*(S^0)$). 

\begin{lem}\label{Lem:Machinery}
\begin{enumerate}
\item There is a (possibly trivial) $H\f_2$-Adams differential (in the ASS for $\mathit{tmf}_*$)
$$d_r(M^{[k_i]}_{\mathit{tmf}}(\alpha)) = \langle P^{-N_{i+1}}_{-N_i} \rangle (M^{[k_{i+1}]}_{\mathit{tmf}}(\alpha)).$$
\item There is an equality of (possibly trivial) elements in $\mathit{tmf}_*$
$$\langle P^{-N_i}_{M} \rangle \overline{(M^{[k_i]}_{\mathit{tmf}}(\alpha))} = \overline{\langle P^{-N_{i+1}}_{M} \rangle (M^{[k_{i+1}]}_{\mathit{tmf}}(\alpha))}.$$
Here, $\overline{\alpha} \in \mathit{tmf}_*$ denotes an element which $\alpha$ detects in the $H\f_2$-based ASS. 
\end{enumerate}
\end{lem}

These provide a method for lifting our calculations into higher Adams filtration. To get started, we will be aided by the following $\mathit{tmf}$-based analog of \cite[Theorem 6.6]{Beh07}.

\begin{lem}
The algebraic $\mathit{tmf}$-based Mahowald invariant agrees with the first nontrivial $H\f_2$-filtered $\mathit{tmf}$-based Mahowald invariant.
\end{lem}

\subsubsection{A simplification using the ASS for $\mathit{tmf}$}
In \cite{Beh07}, Behrens uses Bruner's computations \cite{Bru98} of the algebraic Mahowald invariant of various classes in $Ext^{**}_A(\f_2)$ as a starting point for computing the $H$-filtered (spherical) Mahowald invariant. The algebraic $\mathit{tmf}$-based Mahowald invariants we will need in order to start computing the $\mathit{tmf}$-based Mahowald invariants of the $bo$-based Mahowald invariants of $2^i$ were given in Theorem \ref{atmf1}. To lift these into higher Adams filtration, we need the following facts about the differentials in the ASS for $\mathit{tmf}$. 

\begin{lem}
The differentials in the ASS converging to $\mathit{tmf}_*$ are $v^{32}_2$- and $v^4_1$-linear. 
\end{lem}

\begin{proof}
Multiplication by the permanent cycle $\Delta^8$ detects $v_2^{32}$ in the ASS, so the differentials are $v_2^{32}$-linear. Similarly, multiplication by the permanent cycle $w_1$ detects $v^4_1$ in the ASS, so the differentials are $v_1^{4}$-linear. 
\end{proof}

In view of the $v^8_2$-periodicity of the algebraic $\mathit{tmf}$-based Mahowald invariants and the $v^{32}_2$- and $v^4_1$-periodicity of the differentials in the ASS, it suffices to compute the $\mathit{tmf}$-based Mahowald invariants $M_{\mathit{tmf}}(M_{bo}(2^i))$, $0 \leq i \leq 31$, in order to determine $M_{\mathit{tmf}}(M_{bo}(2^i))$ for all $i \geq 1$. 

\subsection{Psuedo-algorithm for computation}\label{Sec:Pseudoalgorithm}

Before we begin computing, we give a schematic for calculating $M_{\mathit{tmf}}(\alpha)$. Our standard approach is as follows:
\begin{enumerate}
\item Determine the algebraic approximation $M_{\mathit{tmf}}^{alg}(\alpha)$ to the $\mathit{tmf}$-based Mahowald invariant by consulting Theorem \ref{atmf1}. 
\item If this approximation survives in the ASS for $\mathit{tmf}$ \textbf{and} there are no elements of higher Adams filtration which survive in the AHSS for $\mathit{tmf}^{tC_2}$ which can detect $\alpha$, then this is $M_{\mathit{tmf}}(\alpha)$. Otherwise, continue to Step (3). 
\item Check if there is an obvious way of applying Lemma \ref{Lem:Machinery} to lift this approximation into higher Adams filtration.
\begin{enumerate}
\item If yes, lift the element and return to the beginning of Step (2). 
\item If no, and if $\alpha$ does not survive in the ASS for $\mathit{tmf}$, check the AHSS to see which elements lie in higher Adams filtration which can detect $\alpha$. 
\begin{enumerate}
\item If there is only one element, that element is $M_{\mathit{tmf}}(\alpha)$.
\item If there is more than one element, show that all but one of these are $\mathit{tmf}$-based Mahowald invariants of other elements. The remaining element is $M_{\mathit{tmf}}(\alpha)$. 
\end{enumerate}
\item If no, and if $\alpha$ does survive in the $ASS$ for $\mathit{tmf}$, and if there are elements of higher Adams filtration which survive in the AHSS for $\mathit{tmf}^{tC_2}$ which can detect $\alpha$, then try another element. In this case, the elements in higher Adams filtration (or the current approximation) must be ruled out from being $M_{\mathit{tmf}}(\alpha)$ by showing that they are the $\mathit{tmf}$-based Mahowald invariant of different elements $\alpha', \alpha'', \cdots \neq \alpha$. 
\end{enumerate}
\end{enumerate}

\begin{rem2}
This will not always work, and we may employ other \emph{ad hoc} arguments below. For example, we can read several low-dimensional $\mathit{tmf}$-based Mahowald invariants off of the $E_8$-page of the AHSS, so we do not need to employ more elaborate machinery. 
\end{rem2}

\subsection{$v_1$-periodic $\mathit{tmf}$-based Mahowald invariants}\label{Sec:v1perMI}
With the schematic from Section \ref{Sec:Pseudoalgorithm} in mind, our first goal is to show that classes in high Adams filtration are $\mathit{tmf}$-based Mahowald invariants of $v_0$-periodic classes. In particular, we will show that $M_{\mathit{tmf}}(2^{4i+j} \beta^k)$ contains $c_4^i \eta^j \Delta^k$ for $i \geq 1$, $0 \leq j \leq 2$, and $k \geq 0$, and when $j =3$, it contains $2c_6 c_4^{i-1} \Delta^k$. 

\begin{conv}
We will include the cell of $\Sigma P^\infty_{-N}$ where an element in $\pi_*(\mathit{tmf}^{tC_2})$ is detected in square brackets after its $\mathit{tmf}$-based Mahowald invariant. For example, the expression $M_{\mathit{tmf}}(1) = 1[0]$ means that $1 \in \pi_0(\mathit{tmf}^{tC_2})$ is detected by $1 \in \pi_0(\mathit{tmf})$ on the $0$-cell of $\Sigma P^\infty_{-\infty}$. 
\end{conv}

The proof of the following theorem occupies the remainder of this subsection.

\begin{thm}\label{Thm:2beta}
For $i \geq 1$, $j \geq 0$, and $k \in \{0,1\}$, the following $\mathit{tmf}$-based Mahowald invariants hold, possibly up to the addition of elements in higher Adams filtration.
\begin{enumerate}
\item $M_{\mathit{tmf}}(2^{4i}\beta^j\alpha^k) = \Delta^j c^i_4[-16j-8i-4k],$
\item $M_{\mathit{tmf}}(2^{4i+1}\beta^j\alpha^k) = \Delta^j c^i_4 \eta[-1-16j-8i-4k],$
\item $M_{\mathit{tmf}}(2^{4i+2}\beta^j\alpha^k) = \Delta^j c^i_4 \eta^2[-2-16j-8i-4k],$
\item $M_{\mathit{tmf}}(2^{4i+3}\beta^j\alpha^k) = \Delta^j c^{i-1}_4 2c_6[-4-16j-8i-4k].$
\end{enumerate}
\end{thm}

We begin with the cases where $j=k=0$.

\begin{lem}\label{Lem:2i}
For $i \geq 1$, we have the following $\mathit{tmf}$-based Mahowald invariants.
\begin{enumerate}
\item $M_{\mathit{tmf}}(2^{4i}) = c^i_4[-8i],$
\item $M_{\mathit{tmf}}(2^{4i+1}) = c^i_4 \eta[-1-8i],$
\item $M_{\mathit{tmf}}(2^{4i+2}) = c^i_4 \eta^2[-2-8i],$
\item $M_{\mathit{tmf}}(2^{4i+3}) = c^{i-1}_4 2c_6[-4-8i].$
\end{enumerate}
\end{lem}

\begin{proof}
Observe that $M^{alg}_{\mathit{tmf}}(2^{4i}) = M^{[4i]}_{\mathit{tmf}}(2^{4i}) = c^i_4[-8i]$ for all $i \geq 1$. In the ASS for $\mathit{tmf}$, $c_4^i$ is a permanent cycle. Moreover, examination of the AHSS along with Lemma \ref{Lem:J} shows that $c_4^i[-8i]$ survives to the $E_8$-page. There are no classes in higher Adams filtration and higher Atiyah-Hirzebruch filtration which could detect $c_4^i[-8i]$, so by the $\mathit{tmf}$-analog of \cite[Cor. 6.2]{Beh07}, we see that $c_4^i[-8i]$ is the $\mathit{tmf}$-based Mahowald invariant of $2^{4i}$. 

The remaining cases follow from dimensions considerations and examination of the AHSS. More precisely, we have
$$|M_{\mathit{tmf}}(2^{4i})| \leq |M_{\mathit{tmf}}(2^{4i+1})| \leq |M_{\mathit{tmf}}(2^{4i+2})| \leq |M_{\mathit{tmf}}(2^{4i+3})| \leq |M_{\mathit{tmf}}(2^{4i+4})|$$
for all $i \geq 1$. In the AHSS, this implies that the collection of elements $\{M_{\mathit{tmf}}(2^{4i+j})\}_{j=1}^3$ is situated between $M_{\mathit{tmf}}(2^{4i})$ and $M_{\mathit{tmf}}(2^{4i+4})$ for all $i \geq 1$. Since $2^{4i+j} = 2^j \cdot 2^{4i}$, the Adams filtration of $M_{\mathit{tmf}}(2^{4i+j})$ must be at least the Adams filtration of $M_{tmf}(2^{4i})$. The result then follows by checking that $c^i_4 \eta[-1-8i]$, $c^i_4 \eta^2[-2-8i]$, and $c^{i-1}_42c_6[-4-8i]$ are the only elements which could satisfy these conditions on stem, Adams filtration, and Atiyah-Hirzebruch filtration.
\end{proof}

The cases where $j=0$ and $k=1$ require a slight modification:

\begin{lem}
For $i \geq 1$, we have the following $\mathit{tmf}$-based Mahowald invariants.
\begin{enumerate}
\item $M_{\mathit{tmf}}(2^{4i}\alpha) = c^i_4[-4-8i],$
\item $M_{\mathit{tmf}}(2^{4i+1}\alpha) = c^i_4 \eta[-5-8i],$
\item $M_{\mathit{tmf}}(2^{4i+2}\alpha) = c^i_4 \eta^2[-6-8i],$
\item $M_{\mathit{tmf}}(2^{4i+3}\alpha) = c^{i-1}_42c_6[-8-8i].$
\end{enumerate}
\end{lem}

\begin{proof}
We begin with $M^{alg}_{\mathit{tmf}}(2^{4i}\alpha) = M^{[4i+3]}_{\mathit{tmf}}(2^{4i}\alpha) = c_4^i a[-8-8i]$ for all $i \geq 1$. There is an Adams differential $d_2(c_4^i a) = c_4 h_2$ for all $i \geq 0$, from which we conclude that $M^{[4i+5]}_{\mathit{tmf}}(2^{4i}\alpha) = c_4^i[-4-8i]$. Indeed, we have
$$\langle \Sigma P^{-5-8i}_{-9-8i} \rangle c_4^i = h_2 c_4^i.$$
The class $c_4^i$ is a permanent cycle in the ASS for $\mathit{tmf}$ and $c_4^i[-4-8i]$ survives to the $E_\infty$-page of the AHSS. It follows that $c_4^i[-4-8i]$ is the $\mathit{tmf}$-based Mahowald invariant of $2^{4i}\alpha$. 

The remaining cases follow from dimension reasons and examination of the AHSS (compare with the proof of Lemma \ref{Lem:2i}). 
\end{proof}

We now move to the general case. We include a proof for $k=0$; the case $k=1$ is identical up to filtration shifts. To start, we calculate $M_{\mathit{tmf}}(2^{4i+3}\beta^j)$ with $j$ even. 

\begin{lem}\label{Lem:2i3beta}
For all $j \geq 0$ even, we have
$$M_{\mathit{tmf}}(2^3\beta^j) = \Delta^j\eta^3[-3-16j]$$
up to the addition of elements in higher Adams filtration. For all $j \geq 0$ even and $i \geq 1$, we have
$$M_{\mathit{tmf}}(2^{4i+3}\beta^j) = \Delta^j c^{i-1}_4 2c_6[-4-16j-8i]$$
up to the addition of elements in higher Adams filtration.
\end{lem}

\begin{proof}
For all $j$ even, we have $M^{alg}_{\mathit{tmf}}(2^3 \beta^j) = M^{[4j+3]}_{\mathit{tmf}}(2^3 \beta^j) = \Delta^j h_1^3[-3-16j]$. This is a permanent cycle and it survives in the AHSS by Lemma \ref{Lem:J}, so by the $\mathit{tmf}$-based analog of \cite[Thm. 6.1]{Beh07}, we conclude that $M_{\mathit{tmf}}(2^3\beta^j) = \Delta^j \eta^3[-3-16j]$ up to the addition of elements in higher Adams filtration. By the same argument, for all $i \geq 1$ and $j$ even, we have $M_{\mathit{tmf}}(2^{4i+3} \beta^j) = \Delta^j c^{i-1}_4 2c_6[-4-16j-8i]$ up to the addition of elements in higher Adams filtration. 
\end{proof}

Keeping $j$ even but varying the power of $2$, we have the following.

\begin{lem}
Theorem \ref{Thm:2beta} holds for $i \geq 1$ and $j \geq 0$ even.
\end{lem}

\begin{proof}
The argument is similar to the proof of Lemma \ref{Lem:2i}. We obtain dimension bounds on $M_{\mathit{tmf}}(2^{4i+\ell}\beta^j)$, $\ell \neq 3$, by comparing to the dimension of $M_{\mathit{tmf}}(2^{4i+3}\beta^j)$ calculated in Lemma \ref{Lem:2i3beta}. These bounds determine analogous bounds on the Atiyah-Hirzebruch filtration of $M_{\mathit{tmf}}(2^{4i+\ell}\beta^j)$. Examination of the AHSS reveals that the only elements with sufficiently high Adams filtration, i.e. the only elements which could detect $M_{\mathit{tmf}}(2^{4i+\ell}\beta^j)$, are in $\{\Delta^j c_4^i \eta^\ell[-\ell-16j-8i]\}_{\ell=0}^2$. Indeed, every other element satisfying these restrictions on Atiyah-Hirzebruch filtration and Adams filtration is already a $\mathit{tmf}$-based Mahowald invariant. The lemma then follows by matching the $\ell$ in $2^{4i+\ell}\beta^j$ to the $\ell$ in $\Delta^j c_4^i \eta^\ell[-\ell-16j-8i]$. 
\end{proof}


\begin{lem}
Theorem \ref{Thm:2beta} holds for $i \geq 1$ and $j\geq0$ (including $j$ odd).
\end{lem}

\begin{proof}
We argue using dimension and filtration restrictions as above. First observe that $M_{\mathit{tmf}}(2^{4i+m}\beta^{2\ell+1})$ must lie between $M_{\mathit{tmf}}(2^{4i+m}\beta^{2\ell})$ and $M_{\mathit{tmf}}(2^{4i+m}\beta^{2\ell+2})$ in the Atiyah-Hirzebruch spectral sequence, and further, its Adams filtration is restricted by the Adams filtration of these elements. Examination of the AHSS reveals that the only elements in these restricted Atiyah-Hirzebruch and Adams filtrations which are not already $\mathit{tmf}$-based Mahowald invariants are the elements claimed in Theorem \ref{Thm:2beta}. The result follows by matching $2^{4i+m}\beta^{2\ell+1}$ to the $(4i+m)$-th element in this list of elements. 
\end{proof}

\subsection{$\mathit{tmf}$-based Mahowald invariants approximating $\beta^h_i$}\label{Sec:v1i}

We now compute the $\mathit{tmf}$-based Mahowald invariants of classes of the form $M_{bo}(2^{i})$, $i \geq 1$. We begin with the classes of the form $\beta^i = M_{bo}(2^{4i})$. 

\begin{prop}
We have the following $\mathit{tmf}$-based Mahowald invariants:
\begin{enumerate}
\item $M_{\mathit{tmf}}(1) = 1[0],$
\item $M_{\mathit{tmf}}(\beta) = \bar{\kappa}[-12],$
\item $M_{\mathit{tmf}}(\beta^2) = \eta \Delta \bar{\kappa}[-29]$
\item $M_{\mathit{tmf}}(\beta^3) = \eta \Delta \bar{\kappa}^2[-41],$
\item $M_{\mathit{tmf}}(\beta^4) = \eta^2 \Delta^2 \bar{\kappa}^2[-58],$
\item $M_{\mathit{tmf}}(\beta^5) = 2\Delta^4 \bar{\kappa}[-76],$ 
\item $M_{\mathit{tmf}}(\beta^6) = \Delta^4 \kappa^3[-90],$
\item $M_{\mathit{tmf}}(\beta^7) = \nu \Delta^6 \kappa[-105].$
\end{enumerate}
\end{prop}

\begin{proof}
\begin{enumerate}
\item This is clear from the definition.
\item We begin with $M^{alg}_{\mathit{tmf}}(\beta) = M^{[4]}_{\mathit{tmf}}(\beta) = \bar{\kappa}[-12]$. Recall that $\bar{\kappa}$ is detected by $g$ in the ASS. Inspection of the AHSS shows that there are no classes in higher Adams filtration which can detect $\beta $, so we conclude that $M_{\mathit{tmf}}(\beta) = \bar{\kappa}[-12]$. 
\item We begin with $M^{alg}_{\mathit{tmf}}(\beta^2) = M^{[8]}_{\mathit{tmf}}(\beta^2) = \Delta^2[-32]$. There is an Adams differential $d_2(\Delta^2) = h^2_{21} v_1 x_{35}$ from which we conclude $M^{[10]}_{\mathit{tmf}}(\beta^2) = x_{35} h_{21}^2[-29]$. Indeed, we have
$$\langle \Sigma P^{-30}_{-33} \rangle x_{35} h_{21}^2 = v_1 x_{35} h^2_{21}.$$
The class $x_{35} h^2_{21}$ detects $\eta \Delta \bar{\kappa}$. Inspection of the AHSS shows that there are no classes in higher Adams filtration which can detect $\beta^2$, so we conclude that $M_{\mathit{tmf}}(\beta^2) = \eta \Delta \bar{\kappa}[-29]$. 
\item We begin with $M^{alg}_{\mathit{tmf}}(\beta^3) = M^{[12]}_{\mathit{tmf}}(\beta^3) = \Delta^2 \bar{\kappa}[-44]$. There is an Adams differential $d_2(\Delta^2 \bar{\kappa}) = v_1 x_{35} h_{21}^6$ from which we conclude $M^{[14]}(\beta^3) = x_{35} h_{21}^6[-41]$. Indeed, we have
$$\langle \Sigma P^{-42}_{-45} \rangle x_{35} h_{21}^6 = v_1 x_{35} h_{21}^6.$$
The class $x_{35} h_{21}^6$ detects $\eta \Delta \bar{\kappa}^2$. Inspection of the AHSS shows that there are no classes in higher Adams filtration which can detect $\beta^3$, so we conclude that $M_{\mathit{tmf}}(\beta^3) = \eta \Delta \bar{\kappa}^2[-41]$. 
\item We begin with $M^{alg}_{\mathit{tmf}}(\beta^4) = M^{[16]}_{\mathit{tmf}}(\beta^4) = \Delta^4[-64]$. There is an Adams differential $d_3(\Delta^4) = x_{35} h^{12}_{21}$ from which we conclude $M^{[19]}_{\mathit{tmf}}(\beta^4) = x_{35} h^{11}_{21}[-58]$. Indeed, we have
$$\langle \Sigma P^{-59}_{-65} \rangle x_{35} h_{21}^{11} = x_{35} h_{21}^{12}.$$
The class $x_{35} h_{21}^{11}$ detects $\eta^2 \Delta^2 \bar{\kappa}^2$. Inspection of the AHSS shows that there are no classes in higher Adams filtration which can detect $\beta^4$, so we conclude that $M_{\mathit{tmf}}(\beta^4) = \eta^2 \Delta^2 \bar{\kappa}^2[-58]$. 
\item We begin with $M^{alg}_{\mathit{tmf}}(\beta^5) = M^{[20]}_{\mathit{tmf}}(\beta^5) = \Delta^4\bar{\kappa}[-76]$. There is an Adams differential $d_3(\Delta^4 \bar{\kappa}) = x_{35}h^{16}_{21}$ from which we would like to conclude $M^{[23]}_{\mathit{tmf}}(\beta^5) = x_{35}h^{15}_{21}[-70]$. Indeed, we have
$$\langle \Sigma P^{-71}_{-77} \rangle x_{35} h_{21}^{15} = x_{35} h_{21}^{16}.$$
However, the class $x_{35}h_{21}^{15}$ detects $\Delta^4 2 \kappa = \Delta^2 \eta^2 \bar{\kappa}^3$ which does not survive in the AHSS. 

In order to calculate $M_{\mathit{tmf}}(\beta^5)$, we must vary the bifiltration of $\Sigma P^\infty_{-\infty}$. Note that $2 \cdot \Delta^4 \bar{\kappa}$ is nonzero in $\mathit{tmf}_*$ and that $2 \Delta^4 \bar{\kappa}[-76]$ survives in the AHSS. Varying the bifiltration of $\Sigma P^\infty_{-\infty}$ so that Adams filtration $23$ occurs before Atiyah-Hirzebruch filtration $-76$ gives $M_{\mathit{tmf}}^{[23]}(\beta^5) = 2\Delta^4 \bar{\kappa}[-76]$. There are no classes in higher Adams filtration which can detect $\beta^5$, so we conclude that $M_{\mathit{tmf}}(\beta^5) = 2 \Delta^4 \bar{\kappa}[-76]$. 

\item We begin with $M^{alg}_{\mathit{tmf}}(\beta^6) = M^{[24]}_{\mathit{tmf}}(\beta^6) = \Delta^6[-96]$. There is an Adams differential $d_2(\Delta^6) = \Delta^4 h^2_{21} v_1 x_{35}$ from which we conclude $M^{[26]}_{\mathit{tmf}}(\beta^6) = \Delta^4 h^2_{21} x_{35}[-93]$. Indeed, we have
$$\langle \Sigma P_{-97}^{-94} \rangle \Delta^4 x_{35} h^2_{21} = \Delta^4 h^2_{21} v_1 x_{35}.$$
There is a differential $d_3(\Delta^4 x_{35} h_{21}^2) = x_{35}h_{21}^{21}$ from which we conclude that $M^{[27]}_{\mathit{tmf}}(\beta^6) \neq \Delta^4 h^2_{21}x_{35}[-93]$. Inspection of the AHSS shows that the only class in Adams filtration greater than 26 which can detect $\beta^6$ is $\Delta^4 \kappa^3$, so we must have $M_{\mathit{tmf}}(\beta^6) = \Delta^4 \kappa^3[-90]$. 
\item We begin with $M_{\mathit{tmf}}^{alg}(v_1^{28}) = M^{[28]}_{\mathit{tmf}}(\beta^7) = \Delta^6 \bar{\kappa}[-108]$. This class does not survive in the ASS, and inspection of the AHSS shows that the only classes in higher Adams filtration which can detect $\beta^7$ are $\nu \Delta^6 \kappa$, $\nu \Delta^6 \kappa \eta$, and $\nu \Delta^6 \nu$. We will see below that $M_{\mathit{tmf}}(v_1^{28} \eta) = \nu \Delta^6 \kappa \eta$ and $M_{\mathit{tmf}}(\beta^7 \eta^2)  = \nu \Delta^6 \kappa \nu$, so we conclude that $M_{\mathit{tmf}}(v_1^{28}) = \nu \Delta^6 \kappa[-105]$. 
\end{enumerate}
\end{proof}

We now compute the $\mathit{tmf}$-based Mahowald invariants of classes of the form $\beta^i \eta = M_{bo}(2^{4i+1})$.

\begin{prop}
We have the following $\mathit{tmf}$-based Mahowald invariants:
\begin{enumerate}
\item $M_{\mathit{tmf}}(\eta) = \nu[-2],$
\item $M_{\mathit{tmf}}(\beta \eta) = \eta \Delta[-16],$
\item $M_{\mathit{tmf}}(\beta^2 \eta ) = \eta^2 \Delta^2[-33],$
\item $M_{\mathit{tmf}}(\beta^3 \eta) = \eta^2 \Delta^2 \bar{\kappa}[-45],$
\item $M_{\mathit{tmf}}(\beta^4 \eta) = \nu \Delta^4[-66],$
\item $M_{\mathit{tmf}}(\beta^5 \eta) = \eta^2 \Delta^5[-81],$
\item $M_{\mathit{tmf}}(\beta^6 \eta) = \eta^2 \Delta^5 \bar{\kappa}[-93],$
\item $M_{\mathit{tmf}}(\beta^7 \eta) = \nu \Delta^6 \kappa \eta[-105].$
\end{enumerate}
\end{prop}

\begin{proof}
\begin{enumerate}
\item This follows from inspection of the AHSS. 
\item We begin with $M^{alg}_{\mathit{tmf}}(\beta  \eta) = M^{[5]}_{\mathit{tmf}}(\beta  \eta) =  \eta \Delta[-16]$. Inspection of the AHSS shows that there are no classes in higher Adams filtration which can detect $\beta  \eta$, so we conclude that $M_{\mathit{tmf}}(\beta  \eta) = \eta \Delta[-16]$. 
\item We begin with $M^{alg}_{\mathit{tmf}}(\beta^2 \eta) = M^{[9]}_{\mathit{tmf}}(\beta^2 \eta) = \Delta^2 \nu[-34]$. This class does not survive in the correct Atiyah-Hirzebruch filtration to detect $\beta^2 \eta$ in the AHSS, so we know that $M_{\mathit{tmf}}(\beta^2 \eta) \neq \Delta^2 \nu[-34]$. The next two classes which could detect $\beta^2 \eta$ in the AHSS are $\eta \Delta \bar{\kappa}$ and $\eta^2 \Delta^2$. We have shown above that $\eta \Delta \bar{\kappa}$ detects $\beta^2$, so it cannot detect $\beta^2 \eta$. Therefore we conclude that $\eta^2 \Delta^2[-33]$ detects $\beta^2 \eta$.
\item We begin with $M^{alg}_{\mathit{tmf}}(\beta^3 \eta) = M^{[13]}_{\mathit{tmf}}(\beta^3 \eta) = \Delta^3 \eta[-48]$. There is an Adams differential $d_2(\Delta^3 \eta) = v_1 x_{35} h^7_{21}$ from which we conclude $M^{[15]}_{\mathit{tmf}}(\beta^3 \eta) = x_{35} h^7_{21}[-45]$. Indeed, we have
$$\langle\Sigma  P^{-46}_{-49} \rangle x_{35} h^7_{21} = v_1 x_{35} h^7_{21}.$$
The class $x_{35} h^7_{21}$ detects $\eta^2 \Delta^2 \bar{\kappa}$. Inspection of the AHSS shows that there are no classes in higher Adams filtration which can detect $\beta^3 \eta$, so we conclude that $M_{\mathit{tmf}}(\beta^3 \eta) = \eta^2 \Delta^2 \bar{\kappa} [-45]$. 
\item We begin with $M^{alg}_{\mathit{tmf}}(\beta^4 \eta) = M_{\mathit{tmf}}^{[17]}(\beta^4\eta) = \Delta^4 \nu[-66]$. This class survives in the AHSS and ASS. Inspection of the AHSS shows that there are no classes in higher Adams filtration which can detect $\beta^4 \eta$, so we conclude that $M_{\mathit{tmf}}(\beta^4\eta) = \nu \Delta^4[-66]$. 
\item We begin with $M^{alg}_{\mathit{tmf}}(\beta^5 \eta) = M_{\mathit{tmf}}^{[21]}(\beta^5 \eta) = \Delta^5 \eta[-80]$. This class supports an Adams differential $d_3(\Delta^5 \eta) = x_{35} h_{21}^{17}$. Since $|M_{\mathit{tmf}}(\beta^5 \eta)| \geq |M_{\mathit{tmf}}(\beta^5)| = |2\Delta^4 \bar{\kappa}| = 116$, we see that the only classes in higher Adams filtration and high enough stem which could detect $\beta^5 \eta$ in the AHSS are $\eta \Delta^4 \bar{\kappa}[-76]$ and $\eta^2 \Delta^5 [-81]$. Since $\eta^2 \Delta^5 = \eta  \Delta^5 \cdot \eta$, we find that varying the bifiltration of $\Sigma P^\infty_{-\infty}$ so that Adams filtration $24$ occurs before Atiyah-Hirzebruch filtration $-81$ gives $M^{[24]}_{\mathit{tmf}}(\beta^5 \eta) = \eta^2 \Delta^5 [-81]$. There are no classes in higher Adams filtration which can detect $\beta^5 \eta$, so we conclude that $M_{\mathit{tmf}}(\beta^5 \eta) = \eta^2 \Delta^5 [-81]$. 

\item We begin with $M^{alg}_{\mathit{tmf}}(\beta^6 \eta) = M_{\mathit{tmf}}^{[25]}(\beta^6 \eta) = \Delta^6 \nu[-98]$. This class does not survive in the AHSS. The only class in higher Adams filtration which can detect $\beta^6 \eta$ is $\eta^2 \Delta^5 \bar{\kappa}$, so we conclude that $M_{\mathit{tmf}}(\beta^6\eta) = \eta^2 \Delta^5 \bar{\kappa}[-93]$. 
\item We begin with $M^{alg}_{\mathit{tmf}}(\beta^7 \eta) = M_{\mathit{tmf}}^{[29]}(\beta^7 \eta) =\Delta^7 \eta[-112]$. This class does not survive in the ASS. The only classes in higher Adams filtration which can detect $\beta^7 \eta$ are $\nu\Delta^6 \kappa \eta$ and $\nu \Delta^6 \kappa \nu$. We will see below that $M_{\mathit{tmf}}(\beta^7 \eta^2) = \nu \Delta^6 \kappa \nu$, so we conclude that $M_{\mathit{tmf}}(\beta^7 \eta) = \nu \Delta^6 \kappa \eta[-105]$. 
\end{enumerate}
\end{proof}

We now compute the $\mathit{tmf}$-based Mahowald invariants of classes of the form $\beta^i \eta^2 = M_{bo}(2^{4i+2})$.

\begin{prop}
We have the following $\mathit{tmf}$-based Mahowald invariants:
\begin{enumerate}
\item $M_{\mathit{tmf}}(\eta^2) = \nu^2[-4],$
\item $M_{\mathit{tmf}}(\beta \eta^2) = \bar{\kappa} \epsilon[-18],$
\item $M_{\mathit{tmf}}(\beta^2 \eta^2 ) = \nu \Delta^2 \eta^2[-35],$
\item $M_{\mathit{tmf}}(\beta^3 \eta^2) = \eta^3 \Delta^3[-49],$
\item $M_{\mathit{tmf}}(\beta^4 \eta^2) = \Delta^4 \nu^2[-68],$
\item $M_{\mathit{tmf}}(\beta^5 \eta^2) =  \eta \Delta \bar{\kappa}^5[-83],$
\item $M_{\mathit{tmf}}(\beta^6 \eta^2) =  \Delta^6 \eta^3[-97],$
\item $M_{\mathit{tmf}}(\beta^7 \eta^2) = \nu \Delta^6 \kappa \nu[-107].$
\end{enumerate}
\end{prop}

\begin{proof}
\begin{enumerate}
\item This follows from inspection of the AHSS. 
\item We begin with $M^{alg}_{\mathit{tmf}}(\beta  \eta^2) = M^{[6]}_{\mathit{tmf}}(\beta  \eta^2) = \Delta \nu^2[-20]$. There is an Adams differential $d_3(\Delta \nu^2) = (c_4+\epsilon) \bar{\kappa} \eta$ from which we conclude $M^{[8]}(\beta  \eta^2) = (c_4+\epsilon)  \bar{\kappa}[-18]$. Indeed, we have
$$\langle \Sigma P^{-19}_{-21} \rangle (c_4+\epsilon)  \bar{\kappa} = \eta (c_4+\epsilon)  \bar{\kappa}.$$
The class $(c_4+\epsilon)  \bar{\kappa}$ detects $\bar{\kappa} \epsilon.$ Inspection of the AHSS shows that there are no classes in higher Adams filtration which can detect $\beta  \eta^2$, so we conclude $M_{\mathit{tmf}}(\beta  \eta^2) = \bar{\kappa} \epsilon[-18]$. 
\item We begin with $M^{alg}_{\mathit{tmf}}(\beta^2 \eta^2) = M^{[10]}_{\mathit{tmf}}(\beta^2 \eta^2) = \Delta^2 \nu^2[-36]$. This class does not survive in the correct column to detect $\beta^2 \eta^2$ in the AHSS, so $M_{\mathit{tmf}}(\beta^2 \eta^2) \neq \Delta^2 \nu^2$. The only other classes which can detect $\beta^2 \eta^2$ are $\eta^2 \Delta^2$ and $\nu \Delta^2 \eta^2$. The class $\eta^2 \Delta^2$ already detected $\beta^2 \eta$, so we conclude that $\nu \Delta^2 \eta^2[-35]$ detects $\beta^2 \eta^2$. 
\item We begin with $M^{alg}_{\mathit{tmf}}(\beta^3 \eta^2) = M^{[14]}_{\mathit{tmf}}(\beta^3 \eta^2) = \Delta^3\nu^2[-52]$. There is an Adams differential $d_2(\Delta^3 \nu^2) = v_1 x_{35} h^8_{21}$ from which we conclude $M^{[16]}_{\mathit{tmf}}(\beta^3 \eta^2) = x_{35} h^8_{21} + \eta^3\Delta^3$. Indeed, we have 
$$\langle \Sigma P^{-50}_{-53} \rangle (x_{35} h^8_{21} + \eta^3\Delta^3) = v_1 (x_{35} h^8_{21} + \eta^3\Delta^3).$$
The class $x_{35} h^8_{21} + \eta^3\Delta^3$ detects $\eta^3 \Delta^3$. Inspection of the AHSS shows that there are no classes in higher Adams filtration which can detect $\beta^3 \eta^2$, so we conclude that $M_{\mathit{tmf}}(\beta^3 \eta^2) = \eta^3 \Delta^3[-49]$. 
\item We begin with $M_{\mathit{tmf}}^{alg}(\beta^4 \eta^2) = M^{[18]}_{\mathit{tmf}}(\beta^4 \eta^2) = \Delta^4 \nu^2[-68]$. Inspection of the AHSS shows that there are no classes in higher Adams filtration which can detect $\beta^4 \eta^2$, so we conclude that $M_{\mathit{tmf}}(v_1^{16}\eta^2) = \Delta^4 \nu^2[-68]$. 
\item We begin with $M^{alg}_{\mathit{tmf}}(\beta^5 \eta^2) = M^{[22]}_{\mathit{tmf}}(\beta^5 \eta^2) = \Delta^5 \nu^2[-84]$. This class does not survive in the AHSS. The only class in higher Adams filtration which can detect $\beta^5 \eta^2$ is $\eta \Delta \bar{\kappa}^5$, so we conclude that $M_{\mathit{tmf}}(\beta^5 \eta^2) = \eta \Delta \bar{\kappa}^5[-83]$. 
\item We begin with $M^{alg}_{\mathit{tmf}}(\beta^6 \eta^2) = M^{[26]}_{\mathit{tmf}}(\beta^6 \eta^2) = \Delta^6 \nu^2[-100]$. This class does not survive in the AHSS. The only class in higher Adams filtration which can detect $\beta^6 \eta^2$ is $\Delta^6 \eta^3$, so we conclude that $M_{\mathit{tmf}}(\beta^6 \eta^2) = \Delta^6 \eta^3[-97]$. 
\item We begin with $M^{alg}_{\mathit{tmf}}(\beta^7 \eta^2) = M^{[30]}_{\mathit{tmf}}(\beta^7 \eta^2) = \Delta^7 \nu^2[-116]$. This class does not survive in the ASS. The only class in higher Adams filtration which can detect $\beta^7 \eta^2$ is $\nu \Delta^6 \kappa \nu$, so we conclude that $M_{\mathit{tmf}}(\beta^7 \eta^2) = \nu \Delta^6 \kappa \nu[-107]$. 
\end{enumerate}
\end{proof}

We now compute the $\mathit{tmf}$-based Mahowald invariants of classes of the form $\beta^{i-1}\alpha = M_{bo}(2^{4i+3})$. 

\begin{prop}
We have the following $\mathit{tmf}$-based Mahowald invariants:
\begin{enumerate}
\item $M_{\mathit{tmf}}(\alpha) = (c_4 + \epsilon)[-4],$
\item $M_{\mathit{tmf}}(\beta \alpha ) = q[-20],$
\item $M_{\mathit{tmf}}(\beta^2 \alpha) = (c_4 + \epsilon)\Delta^2[-36], $
\item $M_{\mathit{tmf}}(\beta^3 \alpha) = \Delta^3 c_4[-52],$
\item $M_{\mathit{tmf}}(\beta^4 \alpha) = \Delta^4 (c_4 + \epsilon)[-68],$ 
\item $M_{\mathit{tmf}}(\beta^5 \alpha) = \Delta^4 q[-84],$
\item $M_{\mathit{tmf}}(\beta^6 \alpha) =  \Delta^6(c_4 + \epsilon)[-100],$
\item $M_{\mathit{tmf}}(\beta^7 \alpha) = \Delta^7 c_4[-116].$
\end{enumerate}
\end{prop}

\begin{proof}
\begin{enumerate}
\item We begin with $M^{alg}_{\mathit{tmf}}(\alpha) = M^{[3]}_{\mathit{tmf}}(\alpha) = a[-8]$. There is a differential $d_2(a) = \nu( c_4+\epsilon)$ from which we conclude $M^{[5]}_{\mathit{tmf}}(\alpha) = (c_4 + \epsilon)[-4]$. Indeed, we have
$$\langle \Sigma P^{-5}_{-9} \rangle (c_4 + \epsilon) = \nu (c_4 + \epsilon).$$
Inspection of the AHSS shows that there are no classes in higher Adams filtration which can detect $\alpha$, so we conclude that $M_{\mathit{tmf}}(\alpha) = (c_4 + \epsilon)[-4]$. 
\item We begin with $M^{alg}_{\mathit{tmf}}(\beta \alpha) = M^{[7]}_{\mathit{tmf}}(\beta \alpha) = \Delta c[-20]$. The class $\Delta c$ detects $q$. Inspection of the AHSS shows that there are no classes in higher Adams filtration which can detect $\beta \alpha$, so we conclude that $M_{\mathit{tmf}}(\beta \alpha) = q[-20]$. 
\item We begin with $M^{alg}_{\mathit{tmf}}(\beta^2 \alpha) = M^{[11]}_{\mathit{tmf}}(\beta^2 \alpha) = \Delta^2 a[-40]$. There is an Adams differential $d_2(\Delta^2 a) = \nu \Delta^2( c_4  + \epsilon)+ x_{59}$  from which we conclude $M_{\mathit{tmf}}(\beta^2 \alpha) \neq \Delta^2 a[-40]$. 
Inspection of the AHSS shows that the only classes in higher Adams filtration which could detect $\beta^2 \alpha$ are $\Delta^2(c_4+\epsilon)$, $\Delta^2(c_4+\epsilon)\eta$, $\Delta^2 c_4 \eta^2$, and $\Delta^2 2c_6$. Note that $M_{\mathit{tmf}}(2^4\beta^2\alpha) = c_4^2 \Delta^2[-44]$, so the elements $\beta^2\alpha$, $2\beta^2\alpha$, $2^2 \beta^2 \alpha$, and $2^3 \beta^2\alpha$ must be detected by the above elements. By comparing dimensions, we conclude that $M_{\mathit{tmf}}(\beta^2 \alpha) = \Delta^2(c_4 + \epsilon)[-36]$. 
\item We begin with $M^{alg}_{\mathit{tmf}}(\beta^3 \alpha) = M^{[15]}_{\mathit{tmf}}(\beta^3 \alpha) = \Delta^3 c[-52]$. Inspection of the AHSS shows that there are no classes in higher Adams filtration which can detect $\beta^3 \alpha$, so we conclude that $M_{\mathit{tmf}}(\beta^3 \alpha) = \Delta^3 c_4[-52]$. 
\item We begin with $M^{alg}_{\mathit{tmf}}(\beta^4 \alpha) = M^{[19]}_{\mathit{tmf}}(\beta^4 \alpha) = \Delta^4 a[-72]$. There is an Adams differential $d_2(\Delta^4 a) = \Delta^4 \nu (c_4 + \epsilon)$ from which we conclude $M^{[21]}_{\mathit{tmf}}(\beta^4 \alpha) = \Delta^4 (c_4 + \epsilon)[-68]$. Indeed, we have
$$\langle \Sigma  P^{-69}_{-73} \rangle \Delta^4 (c_4 + \epsilon) = \nu \Delta^4 (c_4 + \epsilon).$$
Inspection of the AHSS shows that there are no classes in higher Adams filtration which can detect $\beta^4 \alpha$, so we conclude that $M_{\mathit{tmf}}(\beta^4 \alpha) = \Delta^4 (c_4 + \epsilon)[-68]$. 
\item We begin with $M^{alg}_{\mathit{tmf}}(\beta^5 \alpha) = M^{[23]}_{\mathit{tmf}}(\beta^5 \alpha) = \Delta^5 c[-84]$. The class $\Delta^5 c$ detects $\Delta^4 q$. Inspection of the AHSS shows that there are no classes in higher Adams filtration which can detect $\beta^5 \alpha$, so we conclude that $M_{\mathit{tmf}}(\beta^5 \alpha) = \Delta^4 q[-84]$. 
\item We begin with $M^{alg}_{\mathit{tmf}}(\beta^6 \alpha) = M^{[27]}_{\mathit{tmf}}(\beta^6 \alpha) = \Delta^6 a[-104]$. There is an Adams differential $d_2(\Delta^6 a) = \nu \Delta^6 (c_4 +\epsilon) + x_{165}$ from which we conclude that $M_{\mathit{tmf}}^{[29]}(\beta^6 \alpha) \neq  \Delta^6 a[-104]$. The argument for computing $M_{\mathit{tmf}}(\beta^2a)$ carries over by multiplying everything in sight by $\Delta^4$ to show that  $M_{\mathit{tmf}}(\beta^6 \alpha) = \Delta^6(c_4 + \epsilon)[-100]$. 
\item We begin with $M^{alg}_{\mathit{tmf}}(\beta^7 \alpha) = M^{[31]}_{\mathit{tmf}}(\beta^7 \alpha) = \Delta^7 c[-116]$. Inspection of the AHSS shows that there are no classes in higher Adams filtration which can detect $\beta^7 \alpha$, so we conclude that $M_{\mathit{tmf}}(v_1^{30}) = \Delta^7 c_4[-116]$. We note that there are classes in higher Adams filtration, but they are already the $\mathit{tmf}$-based Mahowald invariants of classes above. 
\end{enumerate}
\end{proof}

\bibliographystyle{plain}
\bibliography{master}

\addresseshere

\newpage

\appendix

\section{The $E_8$-page of the AHSS for $\mathit{tmf}^{tC_2}$}\label{App:AHSS}
In this appendix, we present part the $E_8$-page of the AHSS converging to $\pi_*(\mathit{tmf}^{tC_2})$. This $E_8$-page is $(8,8)$-periodic by the periodicity of the differentials in the AHSS. We show the spectral sequence in the range $0 \leq t \leq 7$ and $-142 \leq s \leq 0$. The $s$-range of each piece is listed below it, with the lower $s$-value closer to the bottom of the page. A more complete chart may be viewed at \url{https://e.math.cornell.edu/people/jdq27/tmfAHSSE8.pdf}. 

\begin{figure}
\centering
\includegraphics[scale=0.9]{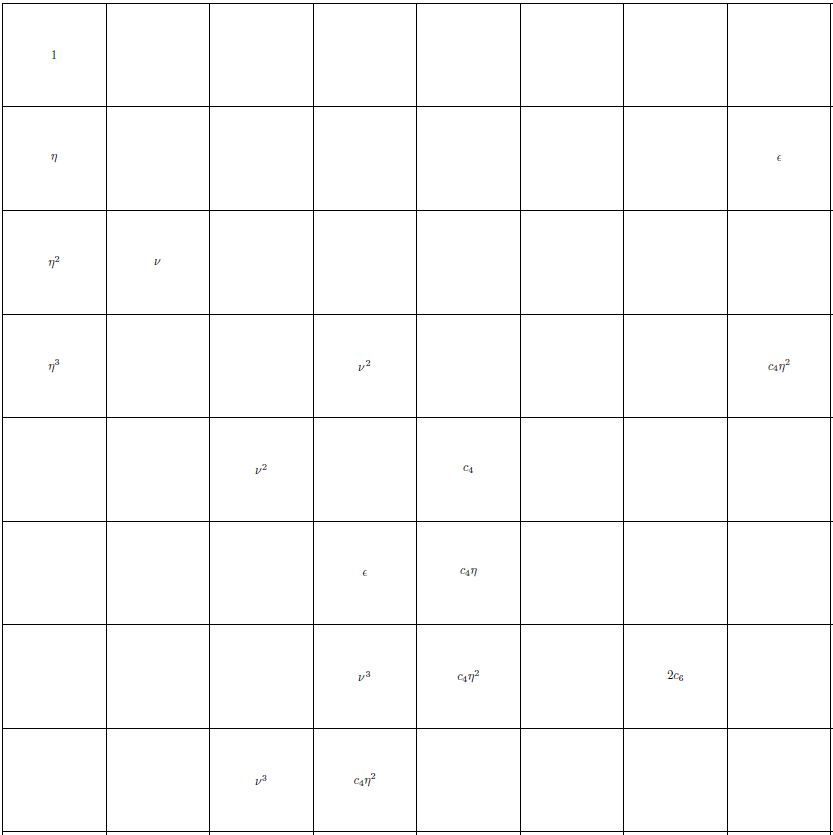} 
\caption{The $E_8$-page of the AHSS for $\mathit{tmf}^{tC_2}$, $0 \leq t \leq 7$, $-7 \leq s \leq 0$. One can read off $M_{\mathit{tmf}}(\alpha)$ for $\alpha \in \{2,2^2,2^3,\eta,\eta^2\}$ from this calculation. Although the $E_8$-page suggests that one could have $M_{\mathit{tmf}}(\eta^3) = \epsilon$, we claim that one can calculate $M_{\mathit{tmf}}(\eta^3) = \nu^3$. Indeed, $dim(M_{\mathit{tmf}}(\eta^3)) \geq dim(M(\eta^3))$ by \cite[Thm. 2.15]{MR93}, and $M(\eta^3) = \nu^3$. }
\end{figure}

\begin{figure}
\centering
\includegraphics[scale=0.9]{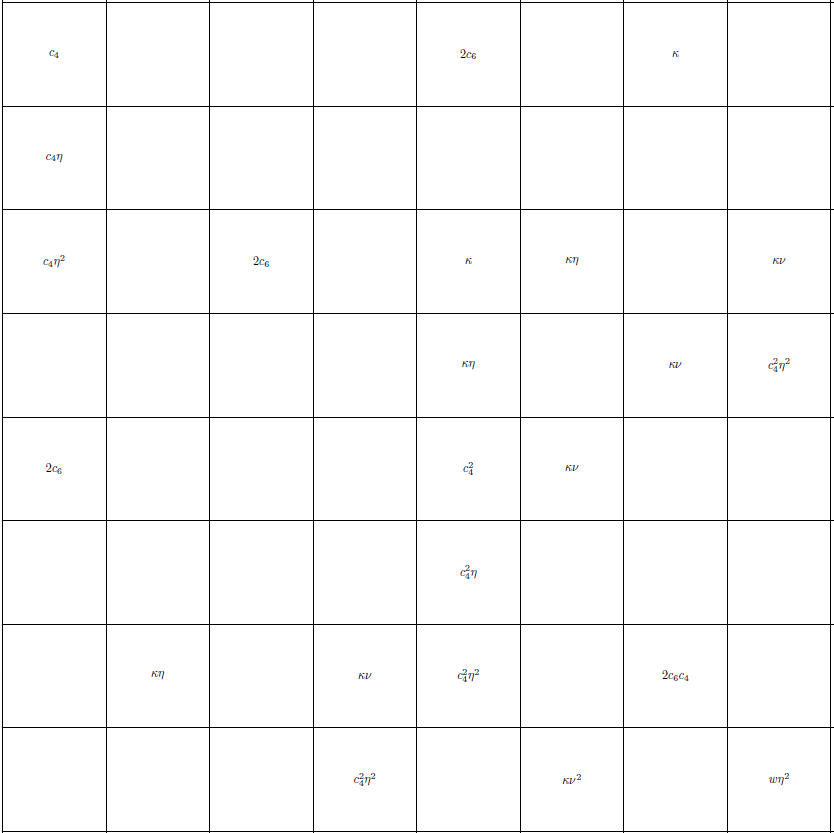} 
\caption{The $E_8$-page of the AHSS for $\mathit{tmf}^{tC_2}$, $0 \leq t \leq 7$, $-15 \leq s \leq -8$. The classes $c_4^i x$ along the column $t=0$ detect $v_1^{4i} \bar{x}$ in $\mathit{tmf}_*$. The class $w \eta^2$ detects $\bar{\kappa} \eta^2$ in $\mathit{tmf}_*$. More generally, the class $w^i x$ detects $\bar{\kappa}^i \bar{x}$ in $\mathit{tmf}_*$. }
\end{figure}

\begin{figure}
\centering
\includegraphics[scale=0.9]{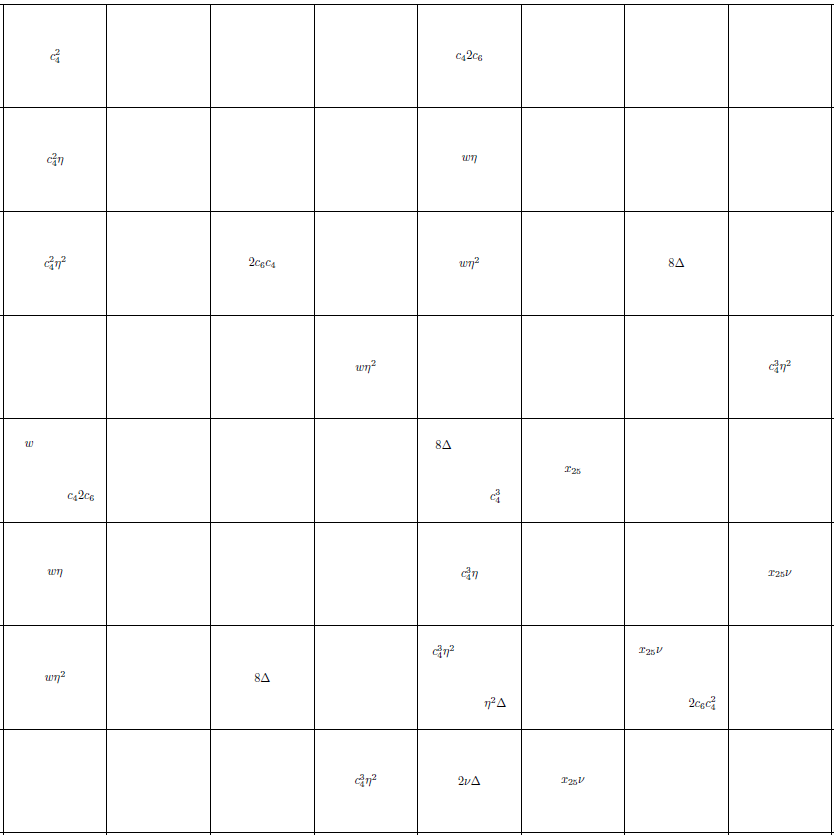} 
\caption{The $E_8$-page of the AHSS for $\mathit{tmf}^{tC_2}$, $0 \leq t \leq 7$, $-23 \leq s \leq -16$. Classes of the form $\Delta^i x$ detect $v_2^{4i} \bar{x}$ in $\mathit{tmf}_*$. The class $x_{25}$ detects $\{ \eta \Delta\}$ in $\mathit{tmf}_{25}$. }
\end{figure}

\begin{figure}
\centering
\includegraphics[scale=0.9]{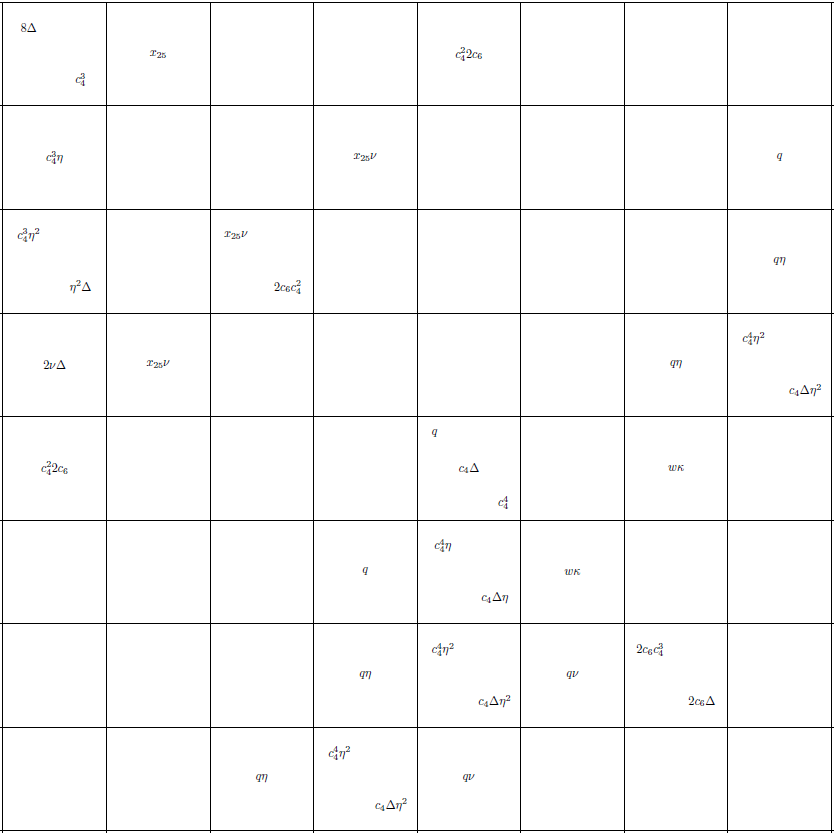} 
\caption{The $E_8$-page of the AHSS for $\mathit{tmf}^{tC_2}$, $0 \leq t \leq 7$, $-31 \leq s \leq -24$.}
\end{figure}

\begin{figure}
\centering
\includegraphics[scale=0.9]{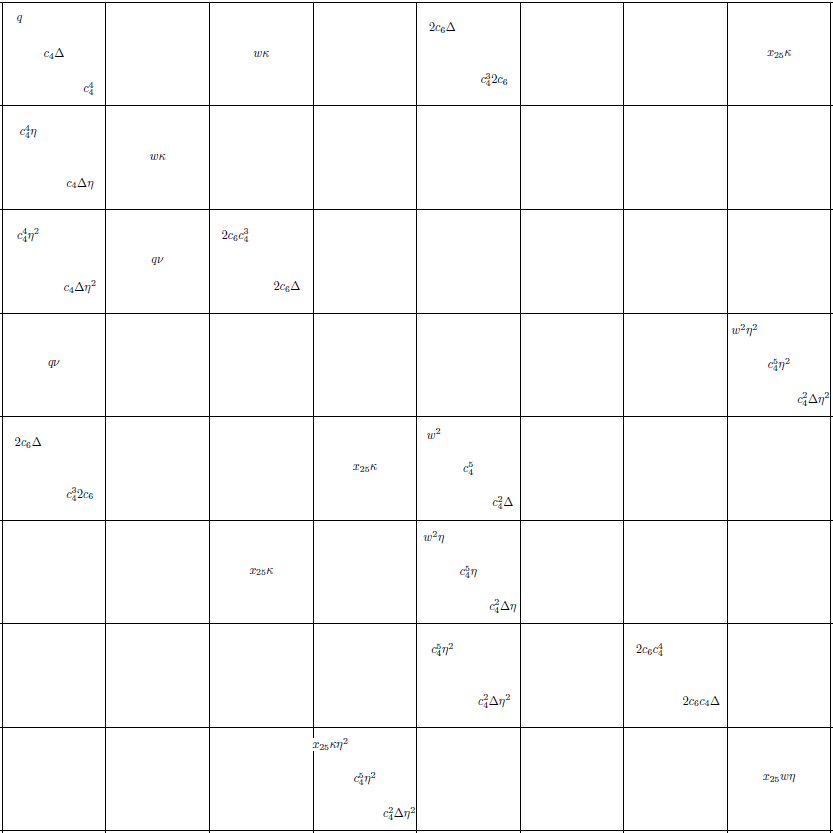} 
\caption{The $E_8$-page of the AHSS for $\mathit{tmf}^{tC_2}$, $0 \leq t \leq 7$, $-39 \leq s \leq -32$.}
\end{figure}

\begin{figure}
\centering
\includegraphics[scale=0.9]{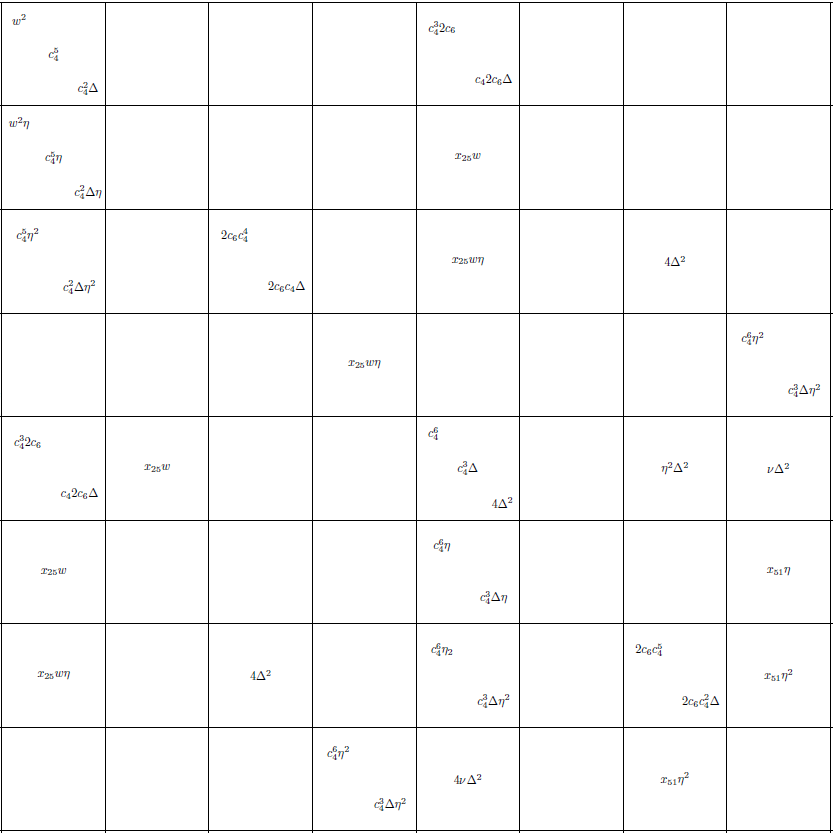} 
\caption{The $E_8$-page of the AHSS for $\mathit{tmf}^{tC_2}$, $0 \leq t \leq 7$, $-47 \leq s \leq -40$. The class $x_{51}$ detects $\{\nu \Delta^2\}$ in $\mathit{tmf}_{51}$. }
\end{figure}

\begin{figure}
\centering
\includegraphics[scale=0.9]{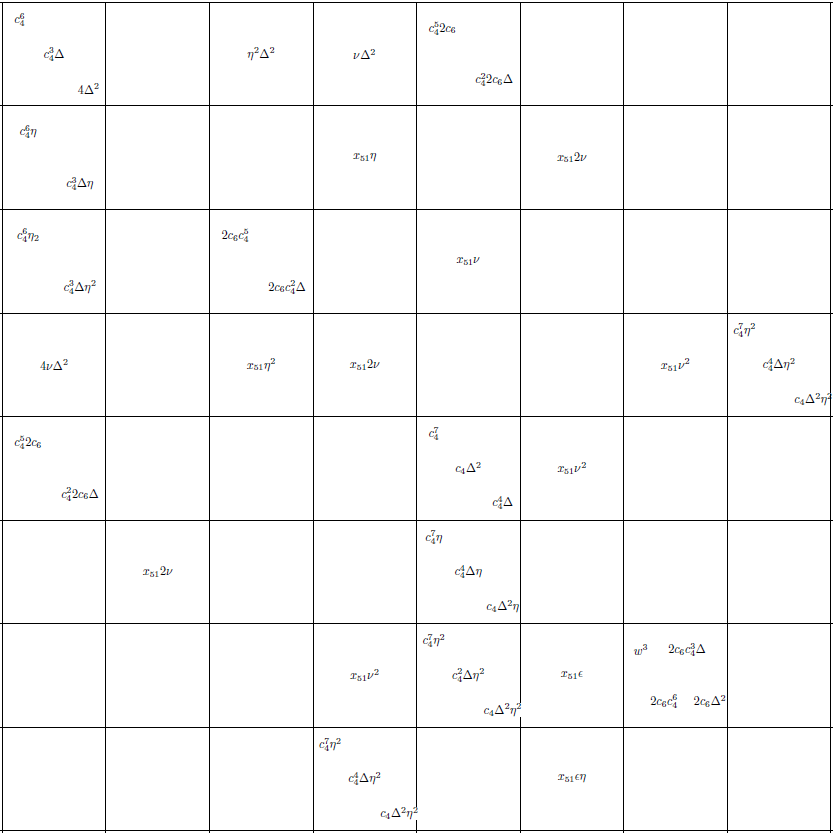} 
\caption{The $E_8$-page of the AHSS for $\mathit{tmf}^{tC_2}$, $0 \leq t \leq 7$, $-55 \leq s \leq -48$.}
\end{figure}

\begin{figure}
\centering
\includegraphics[scale=0.9]{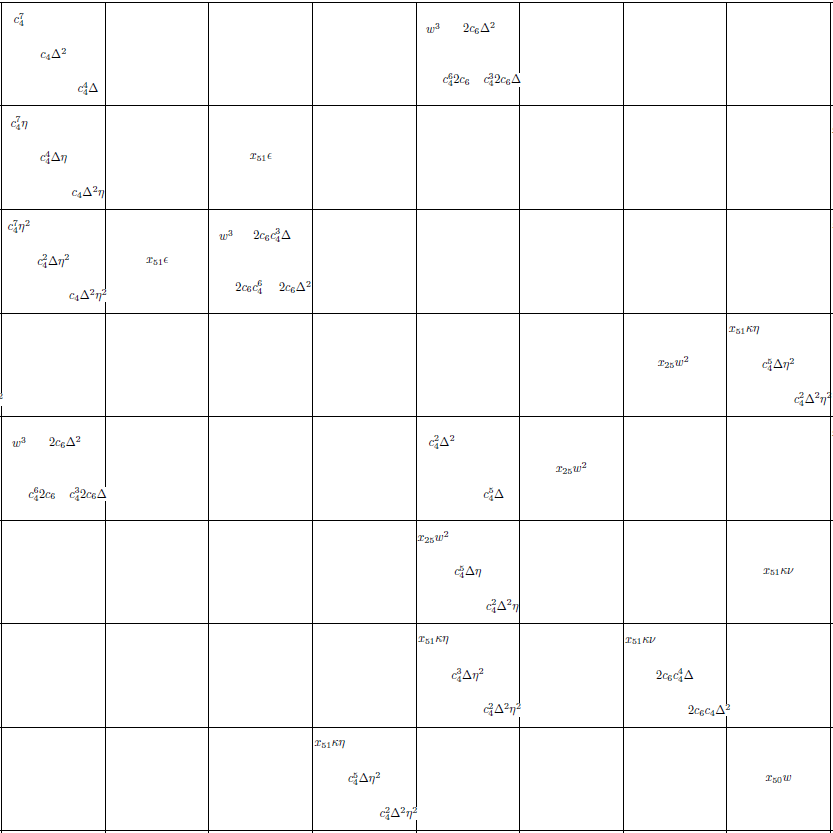} 
\caption{The $E_8$-page of the AHSS for $\mathit{tmf}^{tC_2}$, $0 \leq t \leq 7$, $-61 \leq s \leq -56$.}
\end{figure}

\begin{figure}
\centering
\includegraphics[scale=0.9]{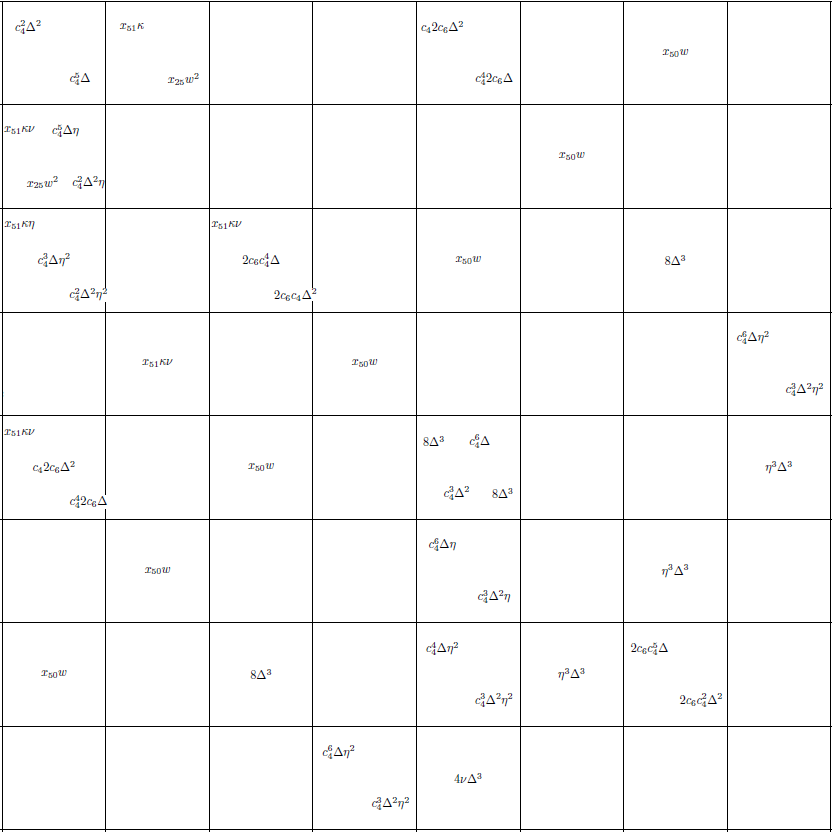} 
\caption{The $E_8$-page of the AHSS for $\mathit{tmf}^{tC_2}$, $0 \leq t \leq 7$, $-69 \leq s \leq -64$. From now on, we suppress terms containing $c_4^i$ for $i \geq 8$ or $c_4^j 2c_6$ for $j \geq 7$ for readability.}
\end{figure}

\begin{figure}
\centering
\includegraphics[scale=0.9]{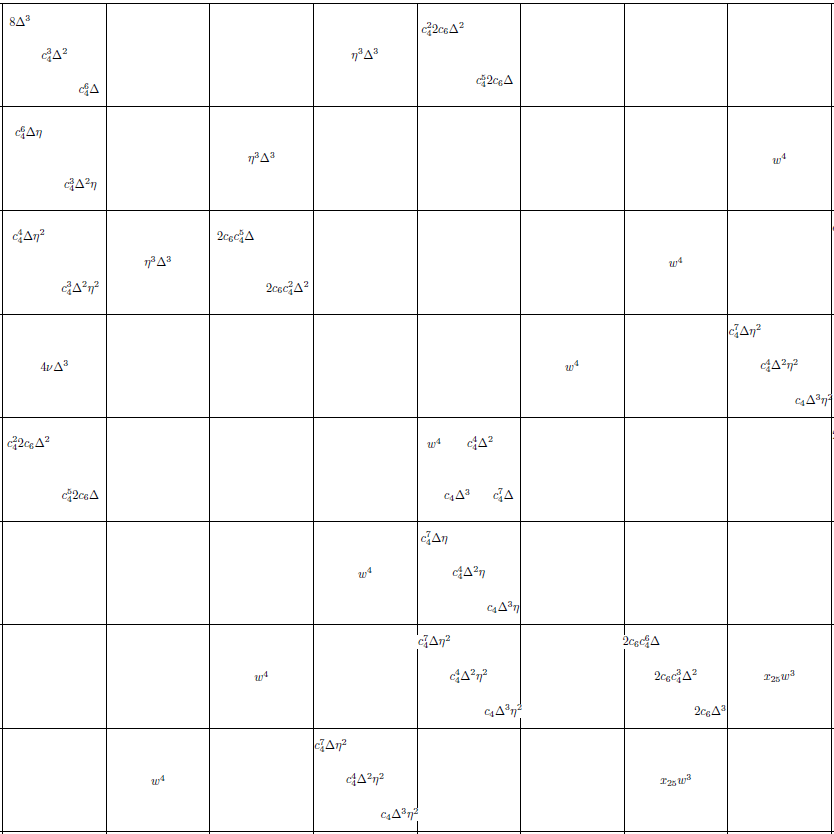} 
\caption{The $E_8$-page of the AHSS for $\mathit{tmf}^{tC_2}$, $0 \leq t \leq 7$, $-79 \leq s \leq -72$.}
\end{figure}

\begin{figure}
\centering
\includegraphics[scale=0.9]{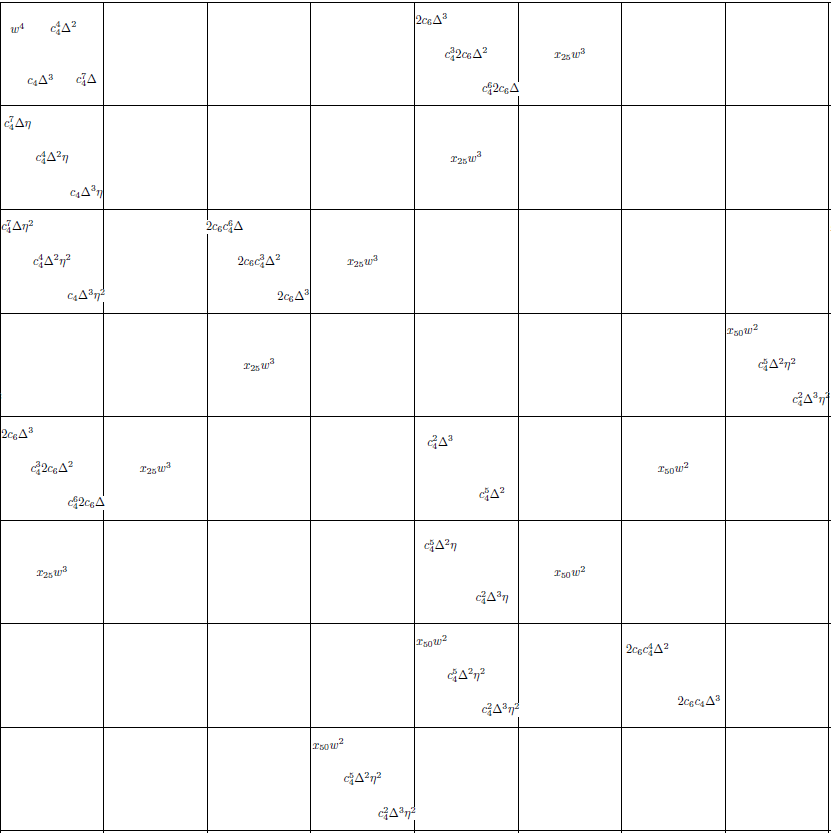} 
\caption{The $E_8$-page of the AHSS for $\mathit{tmf}^{tC_2}$, $0 \leq t \leq 7$, $-87 \leq s \leq -80$.}
\end{figure}

\begin{figure}
\centering
\includegraphics[scale=0.9]{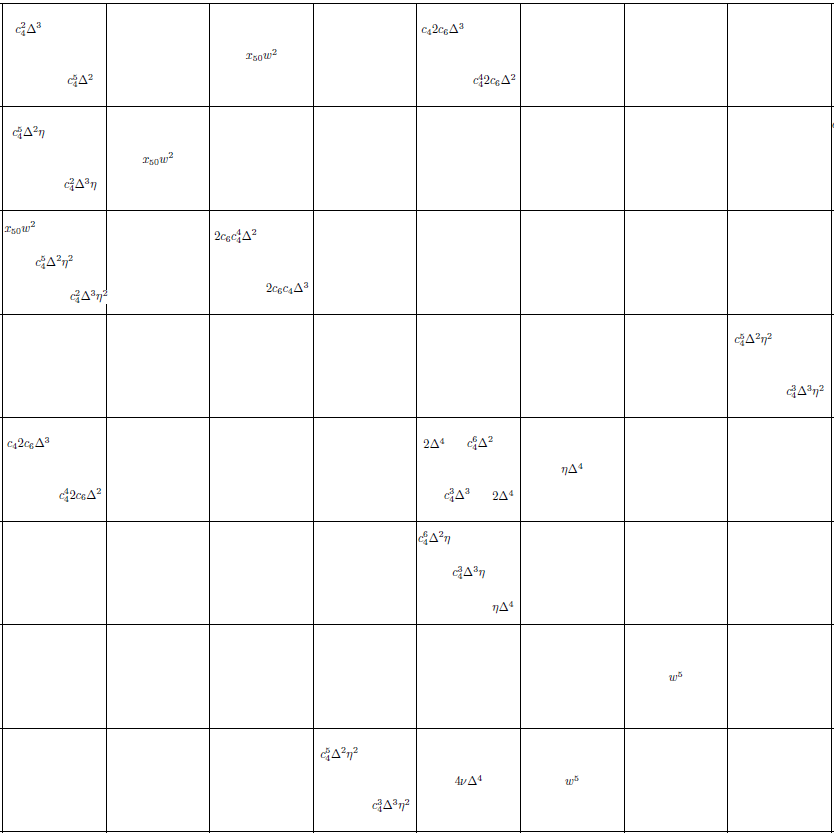} 
\caption{The $E_8$-page of the AHSS for $\mathit{tmf}^{tC_2}$, $0 \leq t \leq 7$, $-95 \leq s \leq -88$.}
\end{figure}

\begin{figure}
\centering
\includegraphics[scale=0.9]{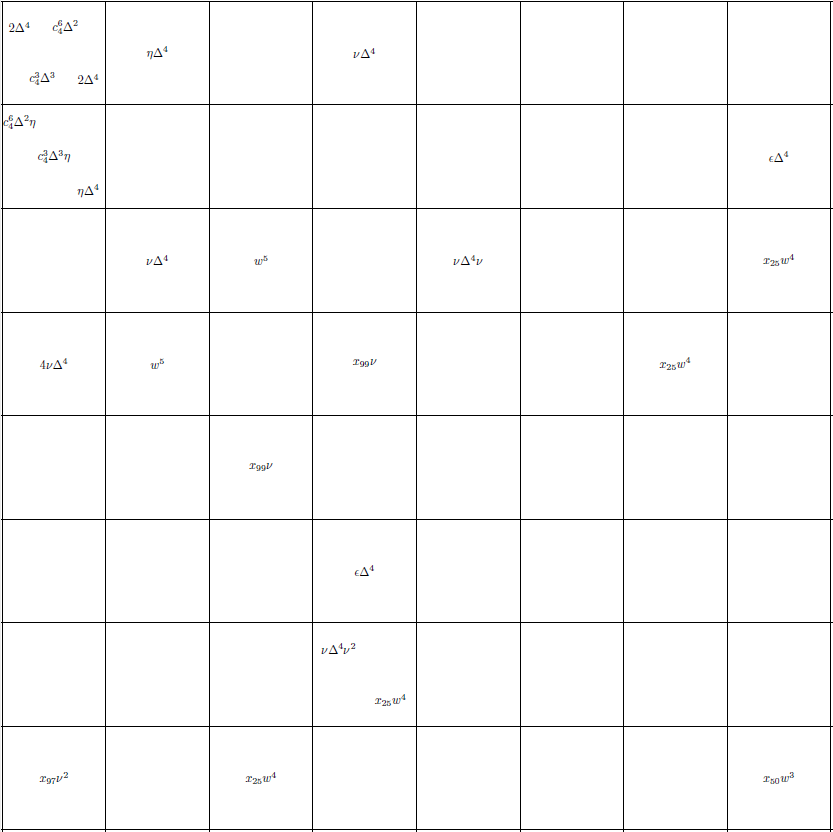} 
\caption{The $E_8$-page of the AHSS for $\mathit{tmf}^{tC_2}$, $0 \leq t \leq 7$, $-103 \leq s \leq -96$. From now on, we suppress most $v_1$-periodic classes. These continue to accumulate as in the previous charts, but we can ignore them in Section \ref{Sec:v1i} in view of Section \ref{Sec:v1perMI}. }
\end{figure}

\begin{figure}
\centering
\includegraphics[scale=0.9]{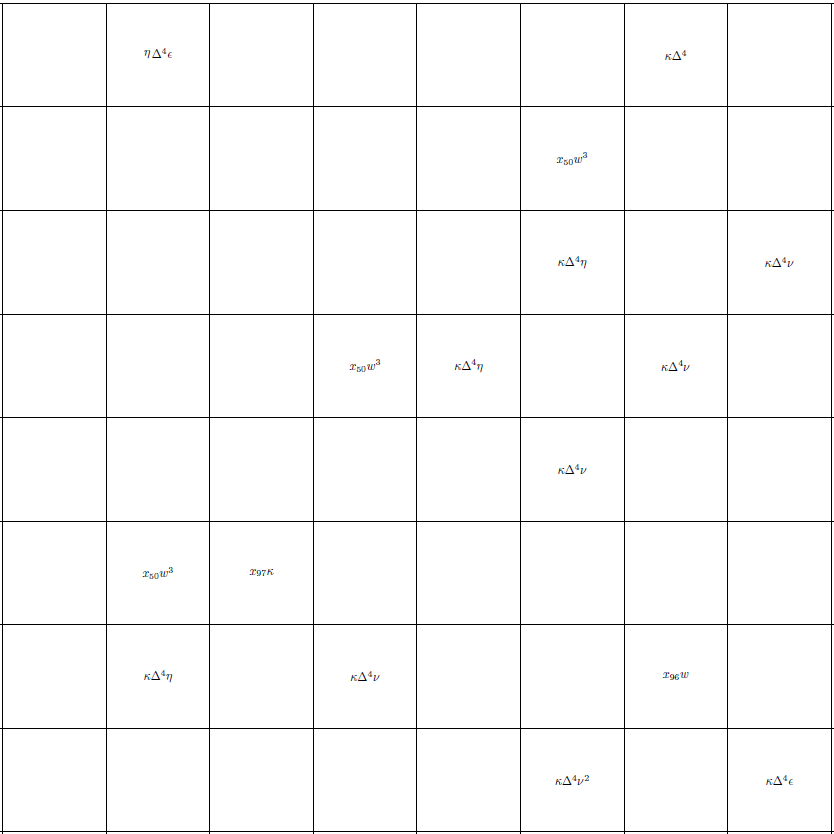} 
\caption{The $E_8$-page of the AHSS for $\mathit{tmf}^{tC_2}$, $0 \leq t \leq 7$, $-111 \leq s \leq -104$.}
\end{figure}

\begin{figure}
\centering
\includegraphics[scale=0.9]{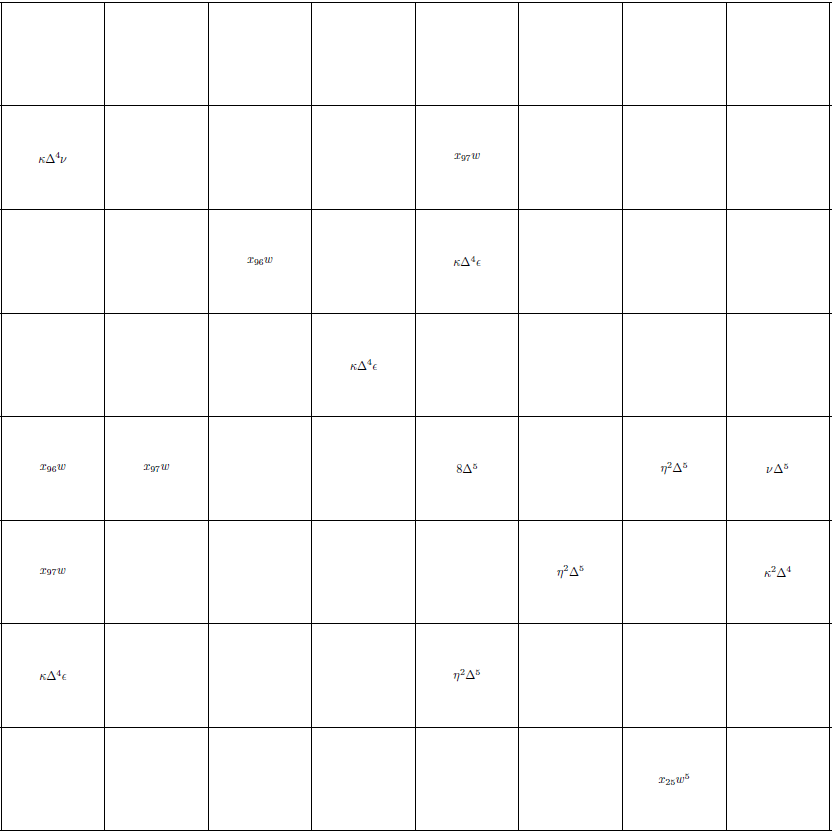} 
\caption{The $E_8$-page of the AHSS for $\mathit{tmf}^{tC_2}$, $0 \leq t \leq 7$, $-119 \leq s \leq -112$.}
\end{figure}

\begin{figure}
\centering
\includegraphics[scale=0.9]{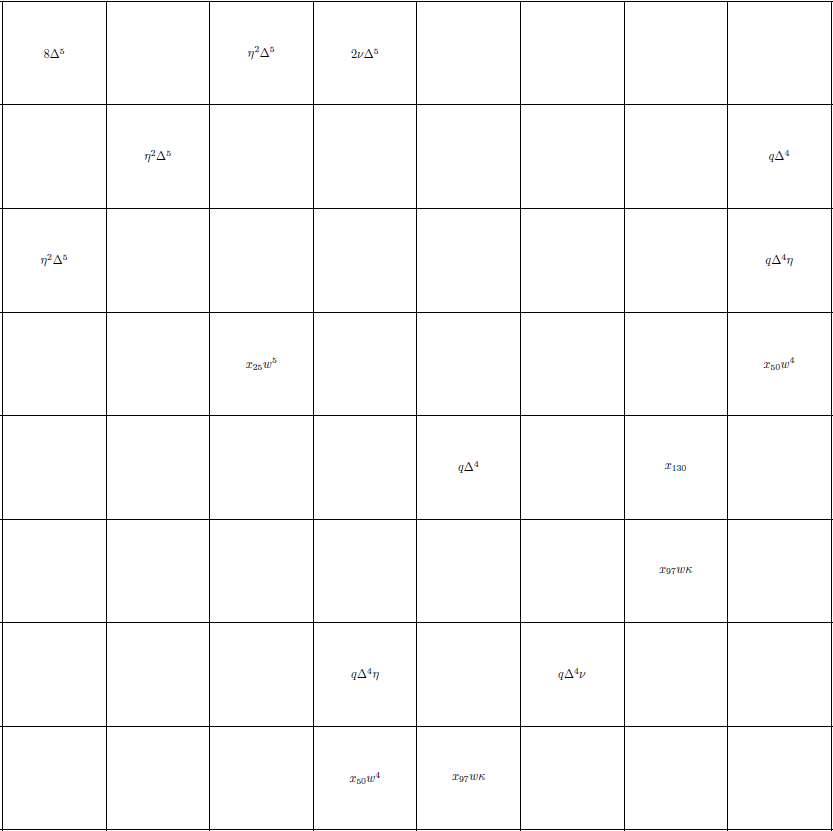} 
\caption{The $E_8$-page of the AHSS for $\mathit{tmf}^{tC_2}$, $0 \leq t \leq 7$, $-127 \leq s \leq -120$.}
\end{figure}

\begin{figure}
\centering
\includegraphics[scale=0.9]{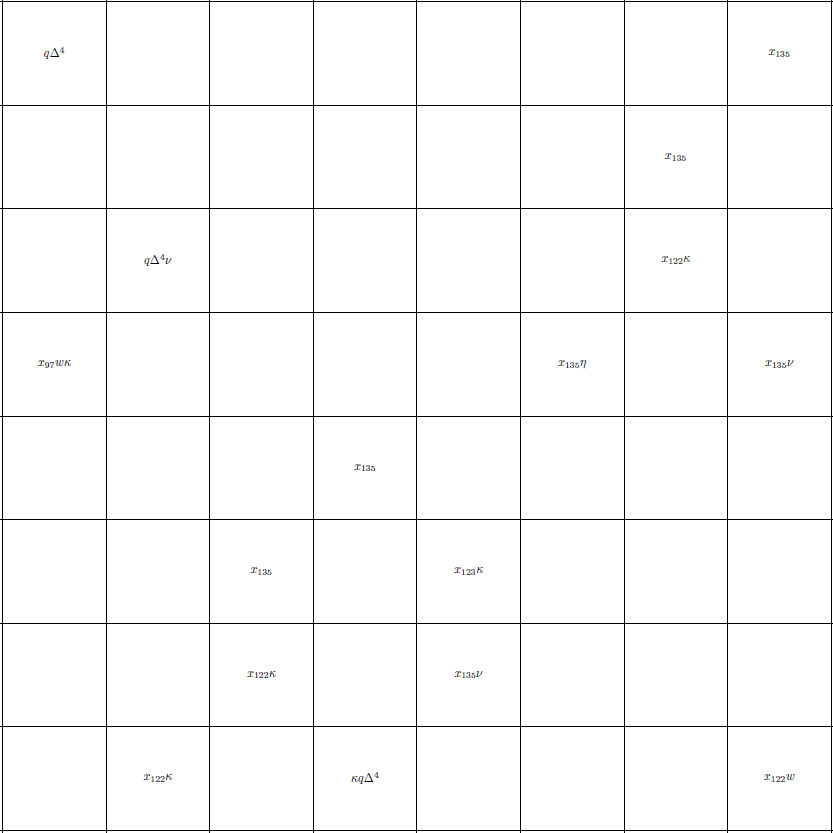} 
\caption{The $E_8$-page of the AHSS for $\mathit{tmf}^{tC_2}$, $0 \leq t \leq 7$, $-135 \leq s \leq -128$.}
\end{figure}

\begin{figure}
\centering
\includegraphics[scale=0.9]{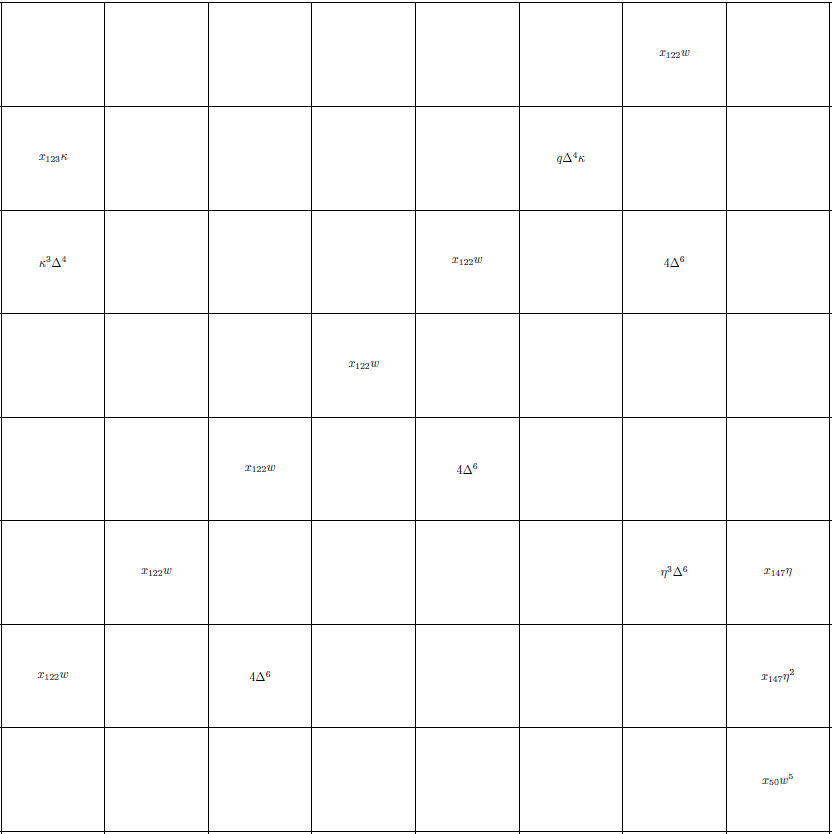} 
\caption{The $E_8$-page of the AHSS for $\mathit{tmf}^{tC_2}$, $0 \leq t \leq 7$, $-143 \leq s \leq -136$.}
\end{figure}

\begin{figure}
\centering
\includegraphics[scale=0.9]{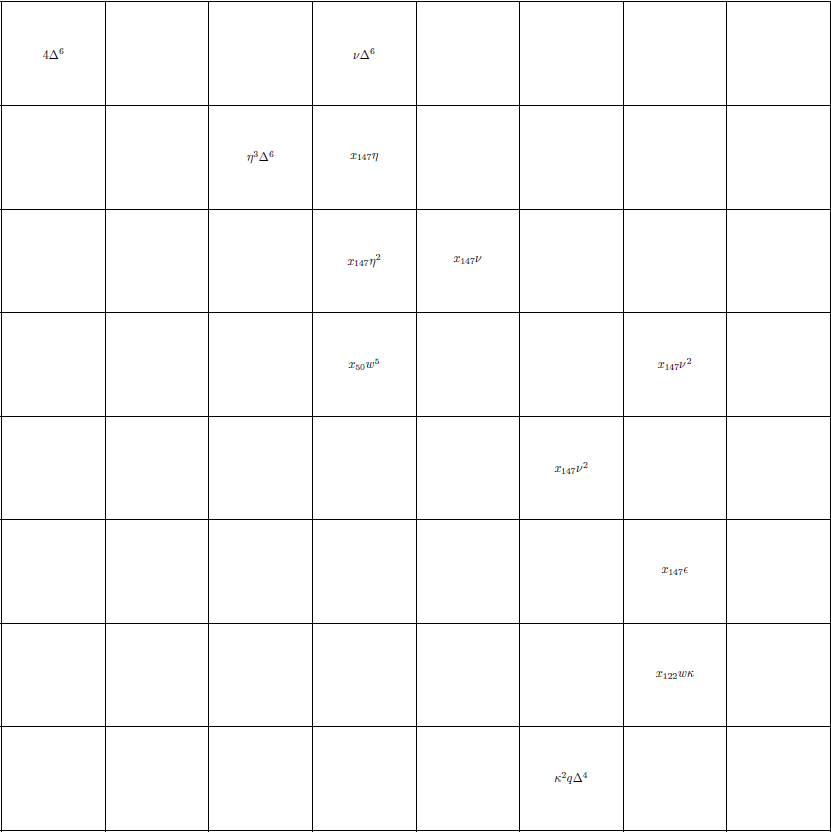} 
\caption{The $E_8$-page of the AHSS for $\mathit{tmf}^{tC_2}$, $0 \leq t \leq 7$, $-151 \leq s \leq -144$.}
\end{figure}

\begin{figure}
\centering
\includegraphics[scale=0.9]{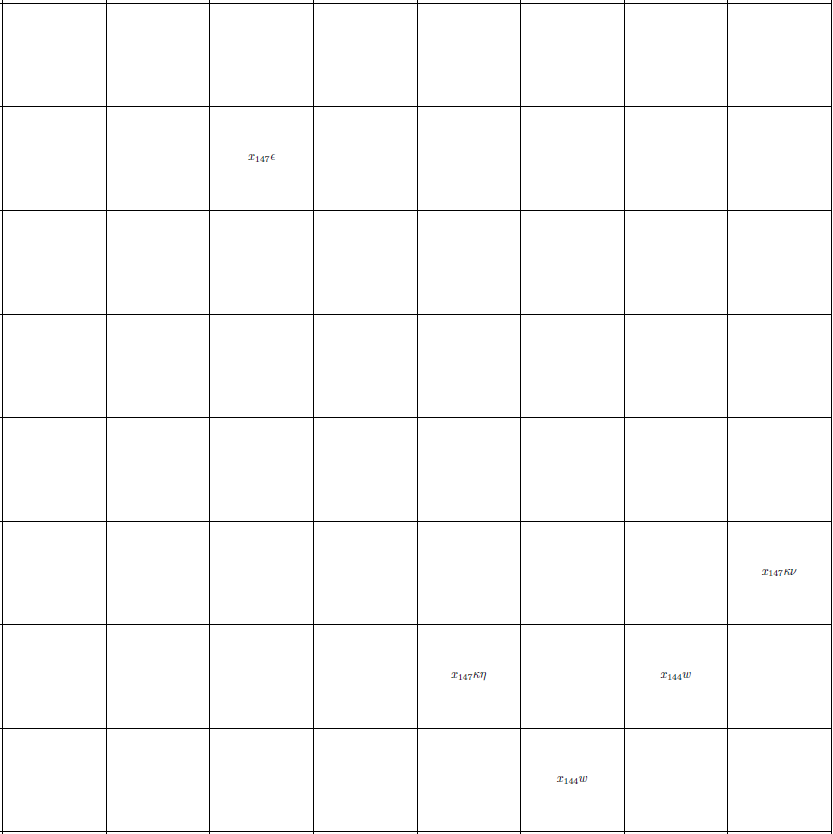} 
\caption{The $E_8$-page of the AHSS for $\mathit{tmf}^{tC_2}$, $0 \leq t \leq 7$, $-159 \leq s \leq -152$.}
\end{figure}

\begin{figure}
\centering
\includegraphics[scale=0.9]{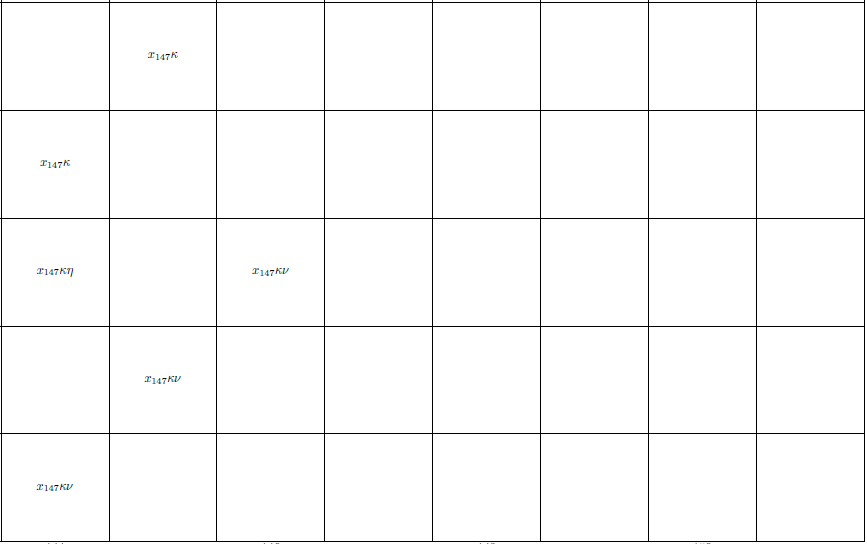} 
\caption{The $E_8$-page of the AHSS for $\mathit{tmf}^{tC_2}$, $0 \leq t \leq 7$, $-164 \leq s \leq -160$.}
\end{figure}

\end{document}